\documentclass[12pt]{amsart}

\usepackage{amssymb}
\usepackage{url}
\usepackage{amscd}
\usepackage{geometry}
\theoremstyle{plain}
\newtheorem{theorem}{Theorem}

\newtheorem{corollary}[theorem]{Corollary}
\newtheorem{lemma}[theorem]{Lemma}
\newtheorem{proposition}[theorem]{Proposition}

\theoremstyle{remark}
\newtheorem{remark}{Remark}
\newtheorem{example}{Example}
\newtheorem{definition}{Definition}
\newtheorem{question}{Question}
\newtheorem{problem}{Problem}

\begin{document}

\date{\today} 

\title{On the Poisson Relation for flat and spherical space forms}

\author{Donato Cianci}

\address{Mathematics Department, Dartmouth College, Hanover, NH 03755}

\email{dcianci@math.dartmouth.edu}

\thanks{The author was partially supported by a GAANN Fellowship at Dartmouth College.}


\begin{abstract}
A general approach to proving that the length spectrum of a compact Riemannian manifold is an invariant of the Laplace spectrum comes from considering the wave trace, a spectrally determined tempered distribution. The Poisson relation states that the singularities of the wave trace can only occur at lengths of closed geodesics. Proving that the Poisson relation is an equality; i.e., every length of a closed geodesic is a singularity of the wave trace, yields an effective means of recovering the length spectrum from the Laplace spectrum.  Regarding spaces of constant curvature this is known to be true for spaces of constant negative curvature. Continuing the study of space forms, we show that the Poisson relation is an equality for all flat manifolds, homogeneous lens spaces and three-dimensional lens spaces with fundamental group of prime order. 
\end{abstract}

\maketitle


\newcommand\cont{\operatorname{cont}}
\newcommand\diff{\operatorname{diff}}

\newcommand{\dvol}{\text{dvol}}

\newcommand{\GL}{\operatorname{GL}}
\newcommand{\myO}{\operatorname{O}}
\newcommand{\myP}{\operatorname{P}}
\newcommand{\eye}{\operatorname{Id}}
\newcommand{\myF}{\operatorname{F}}
\newcommand{\vol}{\operatorname{vol}}
\newcommand{\odd}{\operatorname{odd}}
\newcommand{\even}{\operatorname{even}}
\newcommand{\ol}{\overline}
\newcommand{\mye}{\operatorname{E}}
\newcommand{\myo}{\operatorname{o}}
\newcommand{\myt}{\operatorname{t}}

\newcommand{\re}{\operatorname{Re}}

\newcommand{\scal}{\operatorname{scal}}
\newcommand{\spec}{\operatorname{spec}}
\newcommand{\tr}{\operatorname{trace}}
\newcommand{\sgn}{\operatorname{sgn}}
\newcommand{\SL}{\operatorname{SL}}
\newcommand{\myspan}{\operatorname{span}}
\newcommand{\mydet}{\operatorname{det}}
\newcommand{\SO}{\operatorname{SO}}
\newcommand{\specl}{\operatorname{spec_{\mathcal{L}}}}
\newcommand{\fix}{\operatorname{Fix}}
\newcommand{\id}{\operatorname{id}}
\newcommand{\singsup}{\operatorname{singsupp}}
\newcommand{\wave}{\operatorname{wave}}
\newcommand{\ind}{\operatorname{ind}}
\newcommand{\mynull}{\operatorname{null}}
\newcommand{\floor}[1]{\left \lfloor #1  \right \rfloor}

\newcommand\restr[2]{{
  \left.\kern-\nulldelimiterspace 
  #1 
  \vphantom{\big|} 
  \right|_{#2} 
  }}


\section{Introduction}

The {\it Laplace spectrum} of a closed Riemmanian manifold $(M,g)$, denoted $\spec(M,g)$, is the sequence $0=\lambda_0 < \lambda_1 \leq \lambda_2 \leq ... \nearrow \infty$ of eigenvalues (counting multiplicities) of the Laplace-Beltrami operator, denoted $\Delta_g$, acting on $C^{\infty}(M)$. Spectral geometry is concerned with answering the following: To what extent does the Laplace spectrum determine the geometry of $(M,g)$, and how is that information encoded in the spectrum? There is a great deal of literature demonstrating that the Laplace spectrum does not determine $(M,g)$ completely. However, one still expects that certain classes of manifolds are determined by their spectra, and certain geometric features are spectrally determined. For example, the length spectrum of a manifold, denoted by $\specl (M,g)$, is the collection of lengths of its smoothly closed geodesics, and intuition drawn from the quantum correspondence principle and geometric optics leads one to expect that the length spectrum is encoded in the Laplace spectrum.  

In some specific instances, exact trace formulas can be leveraged to verify this intuition. In the case of flat manifolds, Miatello and Rossetti prove, using a generalization of the Poisson summation formula, that the length spectrum is determined by the Laplace spectrum \cite{rossetti}. Similarly, Pesce verified that the length spectrum is determined by the Laplace spectrum on Heisenberg manifolds \cite{pesce}. Using the Selberg trace formula, Huber showed that the length spectrum of a Riemann surface is determined from the Laplace spectrum (see \cite{mckean}, and the references therein). 

The Poisson summation formula and the Selberg trace formula are exact and useful when they can be defined. In contrast, Chazarain \cite{chaza}, and Duistermaat and Guillemin \cite{duistermaat_spec} employed Fourier integral operator methods to show that the singular support of the {\it trace of the wave group}, referred to as the {\it wave trace}, is a subset of the periods of the geodesic flow on $SM$ (see Section \ref{background}), that is:
\begin{equation*}
\singsup (\tr (\exp(-i t \sqrt{\Delta_g}))) \subseteq \pm \specl(M,g),
\end{equation*}
here $\pm \specl(M,g) = \{ \tau \in \mathbb{R}; \hspace{4 pt} |\tau| \in \specl(M,g) \}$.
This relation is known as the {\it Poisson relation}. Since the wave trace is a spectrally determined tempered distribution, showing that the Poisson relation is an equality would imply that $\specl (M,g)$ is determined by the Laplace spectrum. With this in mind, we state the following problem on which this article focuses:

\begin{problem}
When is the Poisson relation an equality?
\end{problem}

For manifolds that satisfy a technical condition known as the {\it clean intersection hypothesis} (we call such manifolds {\it clean}), it is possible to use an asymptotic expansion of the wave trace at lengths in the length spectrum to determine whether the Poisson relation is an equality. By ruling out cancellations in this asymptotic expansion for sufficiently ``bumpy" metrics, Duistermaat and Guillemin show that generically the Poisson relation is an equality \cite{duistermaat_spec} (see Section \ref{background} for a precise statement and explanation). 

Many well-studied metrics are exceptional in the sense that they do not satisfy the generic condition of bumpiness. Indeed, some of the most natural examples, e.g. Riemannian space forms and other spaces with ``large" symmetry groups, are not bumpy. Despite this fact, one can verify that the Poisson relation is an equality for certain classes of manifolds. It is easily seen that the Poisson relation is an equality on flat tori and compact rank-one symmetric spaces. Using the work of Duistermaat, Kolk, and Varadarajan \cite{duistermaat_kolk_locally_symmetric}, one observes that locally symmetric spaces of non-positive curvature are clean and the Poisson relation is an equality on these spaces (see \cite[sec. 10]{prasad}). Guillemin proved that on manifolds of negative curvature \cite{guill_ross} the Poisson relation is an equality, and Sutton proved that the Poisson relation is an equality for a generic bi-invariant metric on a compact Lie group; in fact, it is always an equality when the bi-invariant metric is induced by the Killing form \cite{sutton_preprint}. Besides the results in this article, this is a complete list of manifolds for which the Poisson relation is known to be an equality. In particular, it is surprising that it is unknown whether the Poisson relation is an equality for manifolds of constant sectional curvature.

In this article we take up the task of studying Problem 1 in the constant curvature setting. As we noted previously, the Poisson relation is known to be an equality for flat tori. We show that this is true for arbitrary Euclidean space forms (Proposition \ref{bieberbach_clean}); i.e., compact quotients of Euclidean space by so-called Bieberbach groups. This provides an alternate proof of the result of Miatello-Rossetti  that the length spectrum of flat space can be recovered from its spectrum \cite{rossetti}. 

For manifolds of non-positive curvature, most of the work is in proving that the manifold satisfies the technical conditions of the clean intersection hypothesis. The cancellation issue in the asymptotic expansion is resolved automatically because the Morse index of a closed geodesic in the non-positively curved setting is always zero. Contrasting this is the case of the spherical space forms, where the Morse index of a closed geodesic can be nonzero. Therefore, manifolds of constant positive curvature are a natural class of manifolds to look for cancellations in the asymptotic expansion. Motivated by this discussion, we pose:

\begin{question}
Do spherical space forms satisfy the clean intersection hypothesis?
\end{question}
\begin{question}
Is the Poisson relation an equality for spherical space forms?
\end{question}
\noindent Towards answering Questions 1 and 2, this article focuses on {\it lens spaces} as a test case. Lens spaces are quotients of $2n-1$-dimensional spheres by the action of a cyclic subgroup of $\SO(2n)$. 

We end this discussion of the Poisson relation by noting that there are no known instances for which the Poisson relation is a strict containment. So that we can make precise statements of our results, we provide a more detailed version of the discussion above.

\subsection{The Trace Formula}
\label{background}
In this section we recall necessary background for introducing the wave trace formula. Let $U_g(t) = \exp\left(-i t \sqrt{\Delta_g}\right)$. Then $U_g(t)$ is a one-parameter family of unitary operators on $L^2(M)$ that generates solutions to the following Cauchy problem:
\begin{equation*}
\begin{cases} \frac{d}{dt} u(t,x) +i \sqrt{\Delta_g} u(t,x) =0\\
u(0,x) = f(x).
\end{cases}
\end{equation*}
This family is known as the {\it wave group} since it is related to solving the wave equation on $(M,g)$. The Schwartz kernel of $U_g$ is given by: 
\begin{equation}
U_g(t,x,y) =  \sum_{j=0}^{\infty} \exp \left(-i t \sqrt{\lambda_j} \right) \varphi_j(x) \overline{\varphi_j(y)},
\end{equation}
 where $\varphi_j$ is a normalized eigenfunction for the Laplace-Beltrami operator corresponding to $\lambda_j$. This family of operators is not of trace class ($U_g(0)$ is the identity operator); however, the trace, parametrized by $t$, can be understood as a tempered distribution after pairing with a Schwartz function. Therefore, we have (in the distributional sense):
\begin{equation*}
\tr\left( U_g(t) \right) = \int_M U_g(t,x,x) \hspace{3 pt} \dvol_g  = \sum_{j = 0}^{\infty} \exp\left(-i t \sqrt{\lambda_j} \right).
\end{equation*} 
We call $\tr\left( U_g(t) \right)$ the {\it wave trace}. The wave trace is clearly a spectral invariant. Chazarain \cite{chaza}, and Duistermaat and Guillemin \cite{duistermaat_spec} showed that the singular support of the wave trace, denoted $\singsup\left( \tr\left( U_g(t) \right) \right)$, is a subset of the periods of the geodesic flow on $SM$:
\begin{equation*}
\singsup\left( \tr\left( U_g(t) \right) \right) \subseteq \pm \specl(M,g).
\end{equation*}
This relation is known as the {\it Poisson Relation}. 

In studying whether the Poisson relation is an equality, Duistermaat and \newline Guillemin used a technical condition on the geodesic flow on $SM$ to recover an asymptotic expansion for the wave trace at periods of the geodesic flow. Let $\Phi_{t} \colon TM \to TM$ denote the time-$t$ map of the geodesic flow on $TM$. We will use the same notation for the geodesic flow restricted to $SM$. If $\tau \in \pm \specl(M,g)$, i.e. $|\tau|$ is the length of a closed geodesic in $(M,g)$, let $\fix(\Phi_{\tau})$ denote the subset of $SM$ fixed by $\Phi_\tau$ (this is necessarily non-empty).

\begin{definition}
We call $\tau \in \pm \specl(M)$ {\it clean} if:
\begin{enumerate}
\item $\fix( \Phi_{\tau} )$ is a finite disjoint union of closed submanifolds $\Theta_1,...,\Theta_r$ of $SM$. 
\item Each $\Theta_i$ is clean in the sense of Bott, i.e. for each $\theta \in \fix( \Phi_{\tau} )$ the fixed point set of $D_\theta \Phi_\tau$ is equal to $T_\theta \fix( \Phi_{\tau} )$. 
\end{enumerate} 
\end{definition}
\noindent Clearly if $\tau$ is clean, then $-\tau$ is also clean. If each $\tau \in \specl(M,g)$ is clean, we say that $(M,g)$ is {\it clean} or $(M,g)$ satisfies the {\it clean intersection hypothesis}. We note that this definition is equivalent to the statement that the energy function on the loop space of $(M,g)$ is a {\it Morse-Bott} function.  

Let $\tau \in \pm \specl(M,g)$ with $(M,g)$ clean. Let $\fix(\Phi_{\tau})$ be a disjoint union of closed connected embedded submanifolds $\Theta_1,...,\Theta_r$. Duistermaat and Guillemin \cite[Section 4]{duistermaat_spec} showed under these hypotheses that there is an open interval $I \subset \mathbb{R}$ such that $I \cap \specl(M,g) = \lbrace \tau \rbrace$ (that is, the length spectrum is discrete) and on $I$ the wave trace may be expressed as:
\begin{equation*}
\tr \left( \exp(-i t \sqrt{\Delta_g}) \right)  = \sum_{j=1}^r \beta_j(t - \tau).
\end{equation*}
Here $\beta_j$ is a distribution on $\mathbb{R}$ with a singularity possible only at 0. Moreover, we have:
\begin{equation*}
\beta_j(t) = \int_{-\infty}^{\infty} \alpha_j(s) \exp(-ist) \text{ ds},
\end{equation*}
with
\begin{equation}
\label{asymp_exp}
\alpha_j(s) \overset{s \to +\infty}{\sim} \left( \frac{s}{2 \pi i} \right)^{(d_j - 1)/2} i^{- \sigma_j} \sum_{k=0}^{\infty} \wave_{j,k}(\tau) s^{-k},
\end{equation}
where $d_j = \text{dim}(\Theta_j)$ and $\sigma_j$ denotes the Morse index (in the free loop space) of a closed geodesic of length $\tau$ in $\Theta_j$ (to see that $\sigma_j$ is indeed the Morse index, see the discussion in \cite[Section 6]{duistermaat_spec}). The coefficients, $\wave_{j,k}(\tau)$, are known as the {\it wave invariants}; they are invariants of the Laplace spectrum. By \eqref{asymp_exp} it is also clear that the wave invariants determine how singular the wave trace is at $\tau$. Duistermaat and Guillemin showed (\cite[Theorem 4.5]{duistermaat_spec}) that $\wave_{j,0}(\tau)$ has the following form:
\begin{equation}
\wave_{j,0}(\tau) = \frac{1}{2\pi} \int_{\Theta_j} \text{ d} \mu_j,
\end{equation}
where $\mu_j$ is a canonical positive density on $\Theta_j$ called the {\it Duistermaat-Guillemin density}. Let $D = \max_j \{ \text{dim}(\Theta_j) \}$. 
Then, to verify that the wave trace has a singularity at $\tau$, it suffices to show that the following sum is non-zero: 
\begin{equation}
\label{zeroth}
 \underset{\dim \Theta_j = D}{\sum_{j=1}^n}i^{-\sigma_j} \wave_{j,0}(\tau).
\end{equation}
Showing that \eqref{zeroth} is non-zero amounts to showing that the highest order term in the the expansion of $\alpha(s) := \sum_j \alpha_j(s)$ is non-zero.  
For clarity of exposition, we present the following example, first presented in \cite[Section 4]{duistermaat_spec}.

\medskip

\begin{example}[The Poisson Relation is Generically an Equality] 
\label{generic_example}
A metric $g$ on $M$ is said to be {\it bumpy} if each smoothly closed geodesic $\gamma$ with respect to $g$ has the property that the space of periodic Jacobi fields along $\gamma$ is spanned by $\gamma'(t)$. Equivalently, $g$ is bumpy if for every $\theta \in SM$ lying on a primitive closed geodesic of length $\tau$, $D_\theta \Phi_\tau \colon T_\theta SM \to T_\theta SM$ has only one eigenvalue that is a root of unity. That eigenvalue is 1, and it occurs with multiplicity 1. 
Denote the set of bumpy metrics by $\mathcal{B}(M)$. Then $\mathcal{B}(M)$ forms a $G_{\delta}$-set in the space of all metrics on $M$ \cite{anosov}. Moreover, let $\mathcal{B}'(M)$ denote the set of bumpy metrics with the property that, up to orientation, there is at most one closed geodesic of a given length. It follows from \cite{kling_takens} that $\mathcal{B}'(M)$ contains a $G_{\delta}$-set (in the space of all metrics on $M$).  

If $g \in \mathcal{B}'(M)$, then it is easily seen that $(M,g)$ is clean. For each $\tau \in \pm \specl(M,g)$, $\fix(\Phi_{\tau})$ consists of two connected components. Moreover, one can show (see \cite[Section 4]{duistermaat_spec}) that the sum in \eqref{zeroth} becomes:
\begin{equation*}
\frac{\tau_0}{\pi} i^{\sigma_{\tau}} \left| \text{det}\left( \eye - D_{\theta} \Phi_{\tau} \right) \right|^{-1/2},
\end{equation*}
where $\tau_0$ is the primitive length of the closed geodesic of length $\tau$. From this, it follows that generically the Poisson relation is an equality.
\end{example}
Flat manifolds and lens spaces are not bumpy; closed geodesics come in large families in the unit tangent bundle.  Although both classes of manifolds are locally symmetric, besides direct verification the author is not aware of any other method for showing that flat manifolds and lens spaces are clean. Indeed, there are examples of locally homogeneous \cite{gornet}, and globally homogeneous \cite{sutton_preprint} manifolds that are not clean.

\medskip

\subsection{Flat manifolds and the Trace Formula} Before stating our results, we recall the construction of flat manifolds. A flat manifold is a quotient of $\mathbb{R}^n$ with the flat Euclidean metric by the action of a {\it Bieberbach group} $\Gamma < \mathbb{E}(n)$, a discrete uniform subgroup of the group of Euclidean motions. Since $\mathbb{E}(n) = \myO(n) \ltimes \mathbb{R}^n$, we will represent every $\gamma \in \Gamma$ (uniquely) by $(B, \myt_v )$ where here $\myt_v$ is translation by $v \in \mathbb{R}^n$.\footnote{Here the action on $\mathbb{R}^n$ is given by: $(B, \myt_v) \cdot x = Bx+v$.}  It is well-known (see for instance \cite{wolf}) that the pure translations in $\Gamma$ form a maximal normal abelian subgroup of finite index which we will denote by $\Lambda$. By identifying translations with the vectors by which one translates, we will identify $\Lambda$ with a full rank lattice in $\mathbb{R}^n$.  

Recall that in the flat setting the Morse index of every closed geodesic is zero. Therefore, to show that the Poisson relation is an equality on flat manifolds, it suffices to show that flat manifolds are clean. We prove:

\begin{proposition}
\label{bieberbach_clean}
Flat manifolds are clean. Consequently, the Poisson relation is an equality on flat manifolds.
\end{proposition}

\subsection{Lens Spaces and the Trace Formula}
Consider $S^{2n-1}$ equipped with a metric of constant sectional curvature 1. 
 Fix $q>2$ and $p_1,...,p_n$ with $p_1 = 1$ and $p_i$ relatively prime to $q$. Set:
\begin{equation*}
\begin{aligned}[c]
\newcommand*{\tempo}{\multicolumn{1}{|c}{0}}
\newcommand*{\tempR}{\multicolumn{1}{|c}{R(\theta_n)}}
\newcommand*{\tempdots}{\multicolumn{1}{|c}{\ddots}}
T = 
\left[
\begin{array}{ccc}
R(\theta_1) &\tempo &\tempo \\ \hline
0 & \tempdots &\tempo \\ \hline
0& \tempo & \tempR
\end{array}\right]
\end{aligned}
; \qquad
\begin{aligned}[c]
R(\theta_i) = 
\left[
\begin{array}{cc}
 \cos\left(\frac{2 \pi}{q} p_i\right) & -\sin \left(\frac{2 \pi}{q} p_i \right) \\ 
\sin \left(\frac{2 \pi}{q} p_i \right) & \cos \left(\frac{2 \pi}{q} p_i \right) \\ 
\end{array}\right]
\end{aligned}.
\end{equation*}
\noindent Then $T$ generates a cyclic group of order $q$. This group acts freely via isometries on $S^{2n-1}$. The {\it lens space}, denoted $L(q;p_1,...,p_n)$, is defined as the quotient manifold $L(q;p_1,...,p_n) = S^{2n-1}/\langle T \rangle$, it is homogeneous precisely when $p_i \equiv \pm 1$ ($\text{mod } q$).  Throughout, we will use $L$ to denote a general lens space.
To show that the Poisson relation is an equality for $L$, it suffices to prove:
\begin{enumerate}
\item[(a)] Lens spaces are clean. 
\item[(b)] Every $\tau \in \pm\specl(L)$ satisfies: 
\begin{equation*}
 \underset{\dim \Theta_j = D}{\sum_{j=1}^n}i^{-\sigma_j} \wave_{j,0}(\tau)\neq 0.
\end{equation*} 
\end{enumerate}
\noindent We accomplish (a) completely.

\begin{theorem}
\label{cleanliness_result}
Lens spaces are clean.
\end{theorem}
\noindent
An easy corollary of Theorem \ref{cleanliness_result} is the following:

\begin{corollary}
\label{homog_result}
The Poisson relation on a homogeneous lens space is an equality.
\end{corollary}
\noindent To verify the Poisson relation for arbitrary lens spaces, we must show that the terms in \eqref{zeroth} do not cancel. Recall that the {\it systole} of a Riemannian manifold is the length of the shortest non-contractible closed geodesic. For lens spaces, the systole coincides with the length of the shortest smoothly closed geodesic. By computing the Morse index of the closed geodesics in a lens space, we are able to show:
\begin{theorem}
\label{systole_result}
The systole of a lens space is in the singular support of the wave trace.
\end{theorem}
\noindent In Section \ref{DG_density_section}, we provide an explicit formula for the highest order wave invariant. Unfortunately, we are unable to determine whether it is non-zero except in the case of certain three-dimensional lens spaces, which establishes the following result. 

\begin{theorem}
\label{sporadic_result}
Let $L(q; p_1,p_2)$ be a 3-dimensional lens space with fundamental group of prime order, then the Poisson relation is an equality on $L(q;p_1,p_2)$. 
\end{theorem}

We suspect that by using the higher order wave invariants (whose formulas are given in \cite{zelditch}) one could resolve completely the issue of whether the Poisson relation is an equality on lens spaces. We plan to take up this task in future work. Besides spheres and real projective space, the lens spaces presented here are the only other spherical space forms on which it is known that the Poisson relation is an equality. 
\begin{remark}
Corollary \ref{homog_result} and Theorem \ref{sporadic_result} show that the length spectrum is a spectral invariant for homogeneous lens spaces and 3-dimensional lens spaces with fundamental group of prime order, respectively. Indeed, since the volume of a Riemannian manifold is an invariant of the Laplace spectrum, then the order of the fundamental group of the lens space is an invariant of the Laplace spectrum. Therefore, we will see (by Corollary \ref{gen_length}) that the length spectrum of a general lens space is an invariant of the Laplace spectrum. 
\end{remark}

\subsection{Outline} This article is organized as follows: In Section \ref{bieberbach_case} we prove that flat manifolds are clean. In Section \ref{closed_geos_sec} we classify the closed geodesics on an arbitrary lens space and show that lens spaces are clean. 
In Section \ref{morse_index_comp} we recall the formula for the Morse index of a smoothly closed geodesic. Using this formula we compute the Morse index of a closed geodesic in a lens space and use this to prove Theorem \ref{systole_result}. In Section \ref{DG_density_section} we compute the Duistermaat-Guillemin density used in \eqref{zeroth} to express the highest order wave invariants. Finally, in Section \ref{3-dcase} we restrict our attention to 3-dimensional lens spaces and prove Theorem \ref{sporadic_result}.    

\medskip
\noindent
{\bf Acknowledgments.} The author is grateful to his advisor, Craig Sutton, for his patience and generous guidance throughout this project. The author also thanks Thomas Shemanske for supplying the proof of Lemma \ref{number_theory}.

\section{Flat manifolds and cleanliness}
\label{bieberbach_case}

Let $\Gamma < \mathbb{E}(n)$ be a Bieberbach group. The corresponding closed flat manifold will be denoted by $M_\Gamma = \mathbb{R}^n/\Gamma$. From the considerations above, we see that $M_\Gamma$ is (finitely) covered by a flat torus (namely, $\mathbb{R}^n/\Lambda$) which we will denote by $M_{\Lambda}$. This is a regular Riemannian covering. We recall some fundamental properties of $M_{\Gamma}$:

\begin{lemma}[\cite{rossetti} Lemma 1.1]
\label{Bieberbach_characterization}
Let $M_{\Gamma}$ be a Bieberbach manifold. Let $\gamma = (B, \myt_b) \in \Gamma$ with $\Lambda$ any lattice stabilized by $B$ and let $p_B$ be the orthogonal projection onto $\ker \left[ B - \eye \right]$. Then:
\begin{enumerate}
\item $p_B(b) \neq 0$. Moreover, $\ker \left[ B - \eye \right] = \text{image} \left[ B - \eye \right]^{\perp}$.

\item $\Lambda^B := \Lambda \cap \ker \left[ B - \eye \right]$ is a lattice in $\ker \left[ B - \eye \right]$.
\end{enumerate}
\end{lemma}

\noindent Miatello and Rossetti \cite{rossetti} first classified the closed geodesics on Bieberbach manifolds. We follow their argument. 

First, notice that any geodesic in $M_\Gamma$ lifts to a line segment in $\mathbb{R}^n$. Moreover, any closed geodesic in $M_\Gamma$ is the projection of a line in $\mathbb{R}^n$ that is translated into itself by some $\gamma \in \Gamma$. Thus, finding closed geodesics amounts to finding lines in $\mathbb{R}^n$ that are fixed by some $\gamma \in \Gamma$. 

Let $(B, \myt_b) = \gamma \in \Gamma$. Given $v \in \mathbb{R}^n$ write:
\begin{equation*}
\begin{aligned}[c]
v = v_+ + v';
\end{aligned}
\begin{aligned}[c]
\hspace{4 pt} v_+ \in \ker \left[ B - \eye \right] \text{ and } v' \in \ker \left[ B - \eye \right]^\perp.
\end{aligned}
\end{equation*}
Let $\myo_\gamma$ be the unique solution of:
\begin{equation*}
\begin{aligned}[c]
\left( B - \eye \right) \myo_\gamma = - b'; 
\end{aligned}
\begin{aligned}[c]
\hspace{4 pt} \myo_\gamma \in \ker \left[ B - \eye \right]^\perp.
\end{aligned}
\end{equation*}
Then we have the following:
\begin{proposition}[\cite{rossetti} Proposition 2.1]
\label{flat_geos_characterization}
With the notation as above:
\begin{enumerate}
    \item $\gamma$ preserves the lines $\myo_\gamma +u + \mathbb{R} b_+$, with $u \in \ker \left[ B - \eye \right]$. Furthermore let $\alpha_{\gamma, u}(t) = \myo_\gamma +u + t b_+$, then $\gamma \alpha_{\gamma, u}(t) = \alpha_{\gamma,u}(t+1)$. Any line in $\mathbb{R}^n$ stabilized by $\gamma$ is of this form.  
    
    \item The geodesic $\alpha_{\gamma,u}$ projects down to a closed geodesic in $\ol{\alpha}_{\gamma, u}(t)$ in $M_\Gamma$ of length $|| b_+ ||$. Any closed geodesic in $M_\Gamma$ is of the form $\ol{\alpha}_{\gamma, u}$, for some $(B, \myt_b) \in \Gamma$ and $u \in \ker \left[ B - \eye \right]$. 
\end{enumerate}
\end{proposition}

\subsection{Fixed point set of the geodesic flow}

Let $\Phi_t$ be the time-$t$ map of the geodesic flow on $TM_{\Gamma}$ (we will use the same notation for the geodesic flow on $SM_{\Gamma}$). Let $\fix (\Phi_t ) \subset SM_\Gamma$ be the fixed point set of the time-$t$ map. Throughout, we will use the fact that $SM_{\Gamma}$ is trivializable. 

\begin{lemma}
\label{bieber_submanifolds}
Let $\tau \in \specl(M_\Gamma)$, then $\fix (\Phi_\tau )$ is a finite disjoint union of closed embedded submanifolds. 
\end{lemma}

\begin{proof}
Let $\phi \colon  \mathbb{R}^n \to M_\Gamma$ be the regular Riemannian covering map. Then by the discussion above, $\phi$ factors into $\phi  = \phi_2 \circ \phi_1$ where $\phi_1 \colon \mathbb{R}^n \to M_\Lambda$ and $\phi_2 \colon M_\Lambda \to M_\Gamma$ (all maps are regular Riemannian covering maps, $\phi_2$ is a finite cover). The derivatives of these maps induce covering maps on the unit tangent bundles of these manifolds.

By Proposition \ref{flat_geos_characterization}, closed geodesics are of the form $\ol{\alpha}_{\gamma, u}$ for some $\gamma \in \Gamma$ and $u$ as above. Thus:
\begin{equation*}
\fix (\Phi_\tau) = \bigcup_{\gamma} D\phi(A_{\gamma}),
\end{equation*}
where for $\gamma = (B, \myt_b)$ we have:
\begin{equation*}
A_{\gamma} = \left \lbrace \left(x, \frac{1}{\tau}b_+ \right); \hspace{4 pt}  x \in \myo_{\gamma} + \ker\left[B - \eye \right] \right \rbrace,
\end{equation*}
and the union is taken over $\gamma$ satisfying $||b_+|| = \tau$. In general there can be infinitely many $\gamma \in \Gamma$ satisfying this constraint. However, each $\myo_{\gamma}$ is uniquely determined by $\gamma \in \Gamma$, a discrete set. Since a fundamental domain for $M_{\Lambda}$ is always compact, we can choose (by fixing a fundamental domain) a finite subset of $\Gamma$, denoted $\widehat{\Gamma}$, such that:
\begin{equation*}
\fix (\Phi_{\tau}) = \bigcup_{\gamma \in \widehat{\Gamma}} D \phi(A_\gamma).
\end{equation*}

Since $\Lambda \cap \ker\left[B - \eye \right]$ is a sub-lattice (by Lemma \ref{Bieberbach_characterization}), and $\left \lbrace A_\gamma \right \rbrace_{\gamma \in \widehat{\Gamma}}$ is a finite collection of disjoint closed embedded (totally geodesic) submanifolds, then $\left \lbrace D\phi_1(A_\gamma) \right \rbrace_{\gamma \in \widehat{\Gamma}}$ is a finite collection of disjoint closed embedded tori in $SM_{\Lambda}$. Indeed, since $\Lambda \cap \ker \left[B - \eye \right]$ is a sub-lattice, then it is easily seen that there exist slice coordinates for $A_{\gamma}$ that descend to slice coordinates for $D \phi_1(A_\gamma)$. And $\left \lbrace D\phi_1(A_\gamma) \right \rbrace_{\gamma \in \widehat{\Gamma}}$ is a disjoint collection by our choice of $\widehat{\Gamma}$.

  Let $\widetilde{\Gamma} = \Gamma/ \Lambda$ be the covering group of $\phi_2 \colon M_{\Lambda} \to M_{\Gamma}$. Then $\widetilde{\Gamma}$ is finite and acts via isometries on $M_{\Lambda}$. Let $\widehat{\Phi}_t$ be the time-$t$ map of the geodesic flow on $SM_{\Lambda}$. 

Set $\overline{A_{\gamma}} = D\phi_1(A_\gamma)$. We need to show that: 
\begin{equation*}
\fix (\Phi_{\tau}) = \bigcup_{\gamma \in \widehat{\Gamma}} D \phi_2(\overline{A_\gamma}),
\end{equation*}
is a disjoint union of embedded submanifolds in $SM_{\Gamma}$. To do this, it suffices to prove that if $\widetilde{\alpha} \in \widetilde{\Gamma}$ and $\widetilde{\alpha}\cdot \overline{A_\gamma} \cap \overline{A_{\gamma'}} \neq \emptyset$ then $\widetilde{\alpha}\cdot \overline{A_\gamma} = \overline{A_{\gamma'}}$, as sets. First, for a fixed $\overline{A_{\gamma}}$ let $\widetilde{\gamma} \in \widetilde{\Gamma}$ be the element of $\widetilde{\Gamma}$ corresponding to the coset $\gamma \Lambda$. Then $\widetilde{\gamma}$ is the unique (possibly trivial) element of the covering group such that: $\widehat{\Phi}_{\tau}(x,v) = \widetilde{\gamma} \cdot (x,v)$ for all $(x,v) \in \overline{A_{\gamma}}$. Moreover, suppose there is a non-trivial $\widetilde{\alpha} \in \widetilde{\Gamma}$ and $(x,v) \in \overline{A_{\gamma}}$ such that $\widetilde{\alpha} \cdot (x,v) \in \overline{A_{\gamma'}}$. Since $\widetilde{\alpha}$ is an isometry of $M_{\Lambda}$ it commutes with the geodesic flow on $SM_{\Lambda}$. Therefore:
\begin{eqnarray*}
\widetilde{\alpha} \widetilde{\gamma} \cdot (x,v) &=& \widetilde{\alpha} \cdot \widehat{\Phi}_{\tau}(x,v) \\
&=& \widehat{\Phi}_{\tau}(\widetilde{\alpha} \cdot (x,v)) \\
&=& \widetilde{\gamma}' \widetilde{\alpha} \cdot(x,v).
\end{eqnarray*}
Since the action is free, this shows that $\widetilde{\alpha} \widetilde{\gamma} = \widetilde{\gamma}' \widetilde{\alpha}$.

Since $SM_{\Lambda}$ is a trivial fiber bundle and every $\overline{A_{\gamma}}$ just has a sub-torus in the first factor and fixed velocity vector in the second factor, to prove that $\widetilde{\alpha} \cdot \overline{A_\gamma} = \overline{A_{\gamma'}}$ it suffices to prove that $\pi \left( \widetilde{\alpha} \cdot \overline{A_{\gamma}} \right) = \pi \left( \overline{A_{\gamma'}} \right)$. Here, $\pi \colon SM_{\Lambda} \to M_{\Lambda}$ is the bundle projection. 

From the definition of $A_{\gamma}$, it is clear that $\pi \left( \overline{A_{\gamma}} \right) = [\myo_{\gamma} + \ker [B - \eye ]]$, where $\gamma = (B, \myt_b)$ and the brackets denote the equivalence class modulo the action of $\Lambda$. Let $\gamma' = (B', \myt_{b'})$ and $\widetilde{\alpha} = (A, \myt_a)$. Since $\widetilde{\alpha} \widetilde{\gamma} = \widetilde{\gamma}' \widetilde{\alpha}$, then $A B = B'A$. Since $\pi (\overline{A_{\gamma}})$ is an (embedded) totally geodesic submanifold inside $M_{\Lambda}$, it suffices to check that $A(\ker \left[ B - \eye \right]) = \ker \left[ B' - \eye \right]$. However, this follows from the fact that $A$ conjugates $B$ to $B'$. 
Thus $\pi(\widetilde{\alpha} \cdot \overline{A_{\gamma}}) = \pi (\overline{A_{\gamma}})$, as required. 
\end{proof}

Given Lemma \ref{bieber_submanifolds}, we now prove Proposition \ref{bieberbach_clean}:

\begin{proof}[Proof of Proposition \ref{bieberbach_clean}] 

By Lemma \ref{bieber_submanifolds}, $\fix(\Phi_\tau)$ is a disjoint union of closed embedded submanifolds for every $|\tau|$ in the length spectrum.  

We need to prove that if $\tau \in  \pm \specl(M_{\Gamma})$ and $\theta \in \fix(\Phi_{\tau})$ then:
\begin{equation*}
\ker  \left[ \eye -  D_\theta \Phi_\tau  \right] = T_\theta \fix(\Phi_{\tau}).
\end{equation*}
One inclusion is obvious. So it suffices to prove:
\begin{equation*}
\ker \left[ \eye -  D_\theta \Phi_\tau \right] \subset T_\theta \fix(\Phi_{\tau}).
\end{equation*}

First, we characterize $T_{\theta} \fix(\Phi_\tau)$. A connected component of $\fix(\Phi_{\tau})$ is of the form: $D \phi (A_{\gamma})$ for some $\gamma \in \Gamma$ (with the same notation as in the proof of Lemma \ref{bieber_submanifolds}). We have a decomposition of $T_{\theta} TM_{\Gamma} $ into a direct sum of horizontal and vertical subspaces: $T_{\theta} TM_{\Gamma} = H(\theta) \oplus V(\theta)$, with $H(\theta)$ and $V(\theta)$ both isomorphic (as vector spaces) to $T_x S^{2n-1}$ under the maps $D_x \pi$ and $K_\theta$, here $K_{\theta}$ is the connection map, see for instance \cite[Section 1.3]{book_geo_flow}.  Since the $A_{\gamma}$'s are subsets of $S \mathbb{R}^n$ with fixed velocity vector and $M_\Gamma$ is flat, it is clear that $T_{\theta} D \phi (A_{\gamma})$ has no vertical component. This yields the following:
\begin{equation*}
T_{\theta} \fix (\Phi_{\tau}) = \ker(B - \eye) \oplus 0,
\end{equation*}  
where we are using the vertical/horizontal splitting, and the identification of $T_x M_{\Gamma}$ with $\mathbb{R}^n$. Here $\theta \in D\phi(A_\gamma)$, where $\gamma = (B, \myt_b) \in \Gamma$. 

Now suppose that $(X,Y) \in T_{\theta}SM_{\Gamma}$ satisfies: $(X,Y) \in \ker \left[ \eye - D_{\theta} \Phi_{\tau} \right]$. Then this sets up the following system of equations for $(X,Y)$:
\begin{equation}
\label{beiber_relation1}
B(X)+ \tau B(Y) = X,
\end{equation}
and
\begin{equation}
\label{beiber_relation2}
B(Y) = Y.
\end{equation} 
Combing these two equations yields:
\begin{equation}
(B - \eye) X = \tau Y. 
\end{equation}
Clearly $Y \in \ker \lbrace B - \eye \rbrace$. However, since $\text{image} \lbrace B - \eye \rbrace^{\perp} = \ker \lbrace B - \eye \rbrace$, then this implies that $Y = 0$ and $ X \in \ker \lbrace B -  \eye \rbrace$, as required. 
\end{proof}

\section{Lens spaces and cleanliness}
\label{closed_geos_sec}

\subsection{Closed Geodesics on Lens Spaces}
In this section we classify all the closed geodesics in a lens space. Let $L =  L(q; p_1,...,p_n)$ be a lens space. 

Let $\phi: S^{2n-1} \to L$ be the (regular Riemannian) covering map. Throughout this article we will identify $S^{2n-1}$ with its image under the inclusion map into $\mathbb{R}^{2n}$. 
Therefore:
\begin{equation*}
S^{2n-1} = \{ x \in \mathbb{R}^{2n} ; \hspace{4 pt} || x|| = 1 \}.
\end{equation*}
We will also identify $TS^{2n-1}$ with its image under the induced embedding of $TS^{2n-1}$ as a submanifold of $\mathbb{R}^{2n} \oplus \mathbb{R}^{2n} \cong T \mathbb{R}^{2n}$. So:
\begin{equation*}
TS^{2n-1} = \{ (x, y) \in \mathbb{R}^{2n} \oplus \mathbb{R}^{2n} ; \hspace{4 pt} ||x|| = 1 \text{ and } \langle x, y \rangle = 0 \}
\end{equation*}

Define the following map $J: S^{2n-1} \to S^{2n-1}$ by:
\begin{equation*}
\begin{aligned}[c]
\newcommand*{\tempo}{\multicolumn{1}{|c}{0}}
\newcommand*{\tempR}{\multicolumn{1}{|c}{J_n}}
\newcommand*{\tempdots}{\multicolumn{1}{|c}{\ddots}}
J = 
\left[
\begin{array}{ccc}
J_1 &\tempo &\tempo \\ \hline
0 & \tempdots &\tempo \\ \hline
0& \tempo & \tempR
\end{array}\right]
\end{aligned}
; \qquad
\begin{aligned}[c]
J_i = \sgn \left(\sin\left(\frac{2 \pi}{q} p_i \right) \right) 
\left[
\begin{array}{cc}
 0 & -1 \\ 
1 & 0 \\ 
\end{array}\right]
\end{aligned}.
\end{equation*}
Since $J(x)$ is orthogonal to $x$, under the above identifications, we may also interpret $J(x)$ as an element of $T_xS^{2n-1}$. If we define $\mathcal{J}: S^{2n-1} \to TS^{2n-1}$ by $\mathcal{J}(x) = (x; \hspace{2 pt} J(x) )$ we see that $\mathcal{J}$ is a section of the bundle map $\pi: TS^{2n-1} \to S^{2n-1}$. 

We start by classifying closed geodesics on homogeneous lens spaces (Theorem \ref{geos}). Then, we classify closed geodesics on a general lens space (Corollary \ref{gen_length}).   Sakai first proved Theorem \ref{geos} in \cite{sakai_lens} and Corollary \ref{gen_length} in \cite{sakai_lens-correction}. The proof of Theorem \ref{geos} we present here differs from Sakai's. However, both proofs are elementary. We note that Sakai's statements of Theorem \ref{geos} in \cite{sakai_lens} and Corollary \ref{gen_length} in \cite{sakai_lens-correction} contain errors in the initial direction of the short geodesics which we have corrected. We include our own proof of Theorem \ref{geos} since the technique will be used throughout the paper. Recall a lens space is homogeneous if and only if the $p_i$ satisfy $\cos\left(\frac{2 \pi}{q} p_i \right) = \cos \left( \frac{2 \pi}{q} \right)$, i.e. $p_i \equiv \pm 1$ (mod $q$) \cite{wolf}.

\begin{theorem} \label{geos} Let $L(q;p_1,...,p_n)$ be a homogeneous lens space, $x = (x_1,...,x_{2n}) \in S^{2n -1}$, and $\overline{x} = \phi(x)$.
 \begin{enumerate} 
 \item If $q$ is odd, then there is a unique, primitive, geodesic of length $\frac{2 \pi}{q}$ based at $\overline{x}$. This geodesic has initial velocity vector $D_x \phi (J(x))$. All other geodesics are closed with primitive length $2 \pi$.  
 \item If $q$ is even, then there is a unique, primitive, geodesic of length $\frac{2 \pi}{q}$ based at $\overline{x}$. This geodesic has initial velocity vector $D_x \phi (J(x))$. All other geodesics are closed with primitive length $\pi$. 
 \end{enumerate}   
\end{theorem}  

\begin{proof}
Since $L$ is covered by a sphere of radius 1 we observe that all geodesics are closed with primitive length at most $2\pi$. In the case that $q$ is even, $L$ is covered by real projective space (equipped with a metric of constant curvature 1) and we conclude that all geodesics are closed with primitive length at most $\pi$.
 
A primitive geodesic in $S^{2n-1}$ based at $x$ with initial direction $y = (y_1,...,y_{2n})$ (here $y \perp x$) is given by: 
\begin{equation*}
\sigma_y(t) = \cos(t) x + \sin(t) y,
\end{equation*}
with $0\leq t \leq 2 \pi$. Observe that $\sigma_y(t)$ descends to a closed geodesic in $L$ of length less than $2 \pi$ ($\pi$, if $q$ is even) 
if and only if the following two equations hold:
\begin{equation}
\label{pos}
T^k x = \cos(t_0) x + \sin(t_0) y,
\end{equation}
and 
\begin{equation}
\label{vel}
T^k y = -\sin(t_0) x + \cos(t_0) y,
\end{equation}
for some $0<t_0<2 \pi$ (or $0<t_0<\pi$ in the case that $q$ is even) and $T^k \in \langle T \rangle$. 
Notice that $T$ commutes with $\pi_i$, where $\pi_i: \mathbb{R}^{2n} \to \mathbb{R}^{2n}$ is the projection onto the subspace $\myspan_{\mathbb{R}} \left( e_{2i-1}, e_{2i} \right)$. Let $x^i = \pi_i(x)$ and $y^i = \pi_i(y)$. Then, \eqref{pos} and \eqref{vel} are equivalent to the following (for each $i$ such that $\pi_i(x)\neq 0$; if $\pi_i(x) = 0$ one sees that $y^i$ is necessarily zero):
\begin{equation}
\label{posi}
\left[R(\theta_i)^k -\cos(t_0) \eye \right] x^i = \sin(t_0) y^i, 
\end{equation}   
and 
\begin{equation}
\label{veli}
\left[R(\theta_i)^k -\cos(t_0) \eye \right] y^i = -\sin(t_0) x^i.
\end{equation}
Computing the determinant of $\left[R(\theta_i)^k -\cos(t_0) \eye \right]$ one sees that necessary and sufficient conditions for $\left[R(\theta_i)^k -\cos(t_0) \eye \right]$ to be singular are $\sin \left(\frac{2 \pi}{q} p_i k \right) = 0$ and $\cos \left(\frac{2 \pi}{q} p_i k \right) = \cos(t_0)$. This occurs only if $t_0 = 0$, $\pi$, or $2\pi$. 

If $q$ is even we may immediately conclude that every short, primitive, closed geodesic based at $\phi(x)$ of length $0<t_0<\pi$ there is a unique direction corresponding to this geodesic. If $q$ is odd, we note that there is no $k$ such that $\cos \left(\frac{2 \pi}{q} p_i k \right) = -1$ (for any $i$). Thus, it is also true in the odd case that primitive closed geodesics of length $0<t_0<2\pi$ are unique. Therefore we conclude that if $t_0$ is a short primitive length of a closed geodesic on $L$ based at a point $\overline{x} = \phi(x)$, then there is a unique direction at $\overline{x}$ corresponding to $t_0$.   

Now we derive necessary conditions for $y^i$. Using \eqref{posi} and \eqref{veli} one can show that $y^i$ satisfies:
\begin{equation}
\left[R(\theta_i)^k -\cos(t_0) \eye \right]^2 y^i = -\sin^2(t_0) y^i.
\end{equation}
Thus $y^i$ is an eigenvector of eigenvalue $-\sin^2(t_0)$. If this is true we must have 
\begin{eqnarray*}
0 &=&\det \left(\left[R(\theta_i)^k -\cos(t_0) \eye \right]^2 + \sin^2(t_0) \eye \right)  \\
&=& 4 \left(\cos\left(\frac{2 \pi}{q} p_i k \right) - \cos(t_0) \right)^2;
\end{eqnarray*}
the second equality follows from standard trigonometric identities.
Therefore, if $t_0$ is the length of a short closed geodesic, then $\cos\left(\frac{2 \pi}{q} p_i k \right) = \cos(t_0)$ for each $i$. Using \eqref{posi} we conclude:
\begin{equation}
\label{torus}
\sgn \left(\frac{\sin\left(\frac{2 \pi}{q} p_i k \right)}{\sin(t_0)} \right)\left[
\begin{array}{cc}
 0 & -1 \\ 
1 & 0 \\ 
\end{array}\right] x^i = y^i.
\end{equation}

This implies that if $\sigma_y(t)$ descends to a primitive short geodesic based at $\phi(x)$ then it is given by the orbit of a one-parameter subgroup of the usual maximal torus $T^n \hookrightarrow \SO(2n)$. 

So far we have not used the fact that $L$ is homogeneous. In this case $\cos \left( \frac{2 \pi}{q} p_i \right) = \cos \left(\frac{2\pi}{q}\right) $ for every $i$. One observes that $\frac{2 \pi}{q}$ is the length of a short geodesic. Using \eqref{torus} we see that all other short lengths $t_0$ satisfying $\cos\left(\frac{2 \pi}{q} p_i k \right) = \cos(t_0)$ are just iterates of the geodesic of length $\frac{2 \pi}{q}$.
\end{proof}

Following \cite{sakai_lens}, we introduce some notation to determine all the closed geodesics in an arbitrary lens space. In the proof of Theorem \ref{geos} we see that  a necessary condition for $\phi(x)$ to have a short geodesic of length $\tau$ is that $\cos \left( \frac{2 \pi }{q} p_i k \right) = \cos(\tau)$ holds for each $i$ such that $\pi^i(x) \neq 0$. With this in mind, we define the following family of equivalence relations on $\lbrace p_1,...,p_n \rbrace$. Fix $l \in \lbrace 1,2,...,q-1 \rbrace$ if $q$ is odd and $l \in \lbrace 1,2,..., \frac{q}{2} - 1 \rbrace$ if $q$ is even. Since each $p_i$ is relatively prime to $q$ there is a unique (mod $q$) integer $k_i$ such that $p_i k_i \equiv l$ (mod $q$). We say $p_i \overset{(l)}{\sim} p_j$ if $\cos\left( \frac{2 \pi}{q} p_i k_j \right) = \cos \left( \frac{2 \pi }{q} l \right)$. $\overset{(l)}{\sim}$ is clearly an equivalence relation. Now let $\lbrace p_1 = p_{j_1},..., p_{j_{m_1}}; ...; p_{j_{m_{b-1}+1}},...,p_{j_{m_b}} \rbrace$ be a partition of $\lbrace p_1,...,p_n \rbrace$ into equivalence classes with respect to $\overset{(l)}{\sim}$. We call a point $\overline{x} \in L(q; p_1,...,p_n)$ with $\overline{x} = \phi(x_1,...,x_{2n})$ {\it $\overset{(l)}{\sim}$-adapted} if there  is an $m_s \in \lbrace m_1,...,m_b \rbrace$ such that $x_{2j-1}=x_{2j}=0$ holds for $p_j \in \lbrace p_1,...,p_n \rbrace \setminus \lbrace p_{j_{m_{s-1}+1}},...,p_{j_{m_s}}\rbrace$. 

Fix $l$ and the decomposition of $\lbrace p_1,...,p_n \rbrace$ into equivalence classes as above. For the $r$-th equivalence class choose $k_r \in \lbrace 1,...,q-1\rbrace$  such that for each \newline $p_i \in \lbrace p_{j_{m_{r-1}+1}},...,p_{j_{m_r}}\rbrace$ we have $\cos \left( \frac{2 \pi}{q} p_i k_r \right) = \cos \left( \frac{2 \pi}{q} l \right)$. There are always two such choices of $k_r$ which determine the orientation of the closed geodesic. We will say that such a choice of $k_r$ {\it realizes} the $r$th equivalence class. Given a choice of $k_r$ for $r=1,...,b$ that realizes each equivalence class, we define the following operator:
\begin{equation*}
\begin{aligned}[c]
\newcommand*{\tempo}{\multicolumn{1}{|c}{0}}
\newcommand*{\tempR}{\multicolumn{1}{|c}{\tilde{J}_n}}
\newcommand*{\tempdots}{\multicolumn{1}{|c}{\ddots}}
J_l = 
\left[
\begin{array}{ccc}
\tilde{J}_1 &\tempo &\tempo \\ \hline
0 & \tempdots &\tempo \\ \hline
0& \tempo & \tempR
\end{array}\right]
\end{aligned}
; \qquad
\begin{aligned}[c]
\tilde{J}_i = \sgn \left(\frac{\sin\left(\frac{2 \pi}{q} p_i k_i \right)}{\sin(2 \pi l/q)} \right)
\left[
\begin{array}{cc}
 0 & -1 \\ 
1 & 0 \\ 
\end{array}\right]
\end{aligned},
\end{equation*}
where here $k_i = k_r$ when $p_i \in \lbrace p_{j_{m_{r-1}+1}},...,p_{j_{m_r}}\rbrace$. As we did with $J$, we define a section of the tangent bundle, $\mathcal{J}_l$, using $J_l$.

With this notation we have the following corollary to Theorem \ref{geos}:

\begin{corollary}
\label{gen_length}
Let $L(q; p_1,...,p_n)$ be a lens space. Let $l \in \lbrace 1,2,...,q-1 \rbrace$ if $q$ is odd and $l \in \lbrace 1,2,..., \frac{q}{2} - 1 \rbrace$ if $q$ is even. Then through every $\overset{(l)}{\sim}$-adapted point $\overline{x} \in L$ with $\phi(x) = \overline{x}$ there exists a geodesic of length $\frac{2 \pi}{q}l$ with initial direction $D_x \phi (J_l(x))$. This closed geodesic is primitive if and only if the base point is never $\overset{(l_i)}{\sim}$-adapted for all divisors $l_i$ of $l$ smaller than $l$. Every closed geodesic of length $\frac{2 \pi}{q} l$ is obtained in this way. All other geodesics are closed with length $2 \pi$ if $q$ is odd and $\pi$ if $q$ is even.  
\end{corollary}   

\begin{proof}
By the argument in Theorem \ref{geos}, we see that $\overline{x} \in L$ has a short, closed geodesic if and only if $\overline{x}$ is $\overset{(l)}{\sim}$-adapted. Moreover, if $x$ is a lift of $\overline{x}$, then by \eqref{torus}, the direction of the short geodesic is given by $J_l(x)$. One observes that if $k$ is another choice which realizes the equivalence class corresponding to $\overline{x}$, then $J_l(x)$ will differ from the original by a sign. Thus, the short geodesic does not depend on the choice of $k$ that realizes each equivalence class. 

If the closed geodesic is not primitive, then $\overline{x}$ must be $\overset{(\tilde{l})}{\sim}$-adapted for some smaller divisor $\tilde{l}$ of $l$. Conversely, if $\overline{x}$ is $\overset{(\tilde{l})}{\sim}$-adapted for some smaller $\tilde{l}$ dividing $l$, then for every $p_i$ in the $\overset{(l)}{\sim}$-class corresponding to $\overline{x}$ there is $\tilde{k}$ such that $p_i \tilde{k} \equiv \pm \tilde{l}$ (mod $q$). Thus, if $\tilde{l} s = l$, then $\tilde{k}s$ realizes the $\overset{(l)}{\sim}$-class corresponding to $\overline{x}$. So $J_{\tilde{l}}(x) = J_{l}(x)$ for $x$ a lift of $\overline{x}$.  
\end{proof}

\subsection{Lens Spaces are Clean}
  With the closed geodesics classified, we tackle the cleanliness issue. Recall that since $\fix (\Phi_{\tau}) = \fix (\Phi_{-\tau})$ and $D_u \Phi_{-\tau} = \left( D_u \Phi_{\tau} \right)^{-1}$, it suffices to check cleanliness for $\tau >0$. We start by showing that for $\tau \in \specl(L)$ the fixed point set of $\Phi_\tau$ is a disjoint union of finitely many closed (embedded) submanifolds.

\begin{lemma}
\label{sub_inhom}
Let $L(q; p_1,...,p_n)$ be a lens space. Suppose $\tau \in \specl(L)$, then $\fix ( \Phi_{\tau} )$ is a finite disjoint union of closed submanifolds. 
\end{lemma}
 
\begin{proof}
It suffices to consider the case where $\tau = \frac{2 \pi}{q} l$ where $l = 1,...,q-1$ in the odd case and $l=1,...,\frac{q}{2}-1$ in the even case. Now, let \newline $\lbrace p_1 = p_{j_1},..., p_{j_{m_1}}; ...; p_{j_{m_{b-1}+1}},...,p_{j_{m_b}} \rbrace$ be a partition of $\lbrace p_1,...,p_n \rbrace$ into equivalence classes with respect to $\overset{(l)}{\sim}$. For the $r-$th equivalence class let $\iota_{r}\colon \mathbb{R}^{m_r} \to \mathbb{R}^{2n}$ be the inclusion of $\mathbb{R}^{m_r}$ into $\mathbb{R}^{2n}$ whose image satisfies:
\begin{multline*}
\text{image}(\iota_r) = \lbrace (x_1,...,x_{2n}) \in \mathbb{R}^{2n}; \hspace{4 pt} x_{2k-1}=x_{2k}=0 \mbox{ holds for } \\
p_k \in \lbrace p_1,...,p_n \rbrace \setminus \lbrace p_{j_{m_{r-1}+1}},...,p_{j_{m_r}}\rbrace \rbrace.
\end{multline*}
 For distinct $r$-values the images of $\iota_r$ are disjoint. Through restriction we get an isometric embedding $\iota_r\colon S^{2m_r-1} \to S^{2n-1}$. Also (after a choice of $k$ realizing each equivalence class) we get a pair of embeddings $\mathcal{J}_l^{\pm} \circ \iota_r\colon S^{2m_r-1} \to S(S^{2n-1})$, which give the lift of the closed geodesics and their time reversals. Here $\mathcal{J}_l^{\pm}(x) = (x, \pm J_l(x))$.

Let the image of $S^{2m_r - 1}$ under $\mathcal{J}_l^{\pm} \circ \iota_r$ be $N_r^{\pm}$ (a pair of disjoint closed embedded submanifolds in $S(S^{2n-1})$). By Corollary \ref{gen_length} we see that both $D \phi(N_r^{\pm})$ are contained in $\fix ( \Phi_{\tau} )$ and are embedded submanifolds of $SL$. Moreover, every element of $\fix ( \Phi_{\tau} )$ is in the image of one of these embeddings. Clearly, the images of $\mathcal{J}_l^{\pm} \circ \iota_r$ are distinct for different $r$ values. Moreover, since $T$ preserves $N_r^{\pm} \subset S(S^{2n-1})$ we see that $D \phi (N_r^{\pm})$ are disjoint submanifolds of $SL$.  
\end{proof} 
\begin{remark}
In case $L$ is homogeneous we can prove Lemma \ref{sub_inhom} by noting that $\fix(\Phi_\tau)$ is just a pair of orbits of a connected compact Lie group (namely the one acting transitively on $L$). Therefore the fixed point set is a disjoint union of two closed submanifolds. 
\end{remark}  

\begin{proposition}
\label{clean}
For each $\tau \in \specl(L(q; p_1,...,p_n))$ and for each $\overline{\theta} \in \fix(\Phi_\tau)$ we have:
\begin{equation*}
\ker \left[ D_{\overline{\theta}} \Phi_\tau - \eye \right] = T_{\overline{\theta}} \fix(\Phi_\tau).
\end{equation*} 
\end{proposition}

Before proving Proposition \ref{clean} we characterize $T_{\overline{\theta}} \fix(\Phi_\tau)$. Fix $\overline{\theta} = (\overline{x}, \hspace{2 pt} \overline{v}) \in SL$. Let $\theta = (x, \hspace{2 pt} v)$ be a lift of $\overline{\theta}$. If $\overline{\theta} \in \fix(\Phi_{\tau})$, then $\theta \in N_r^{\pm}$ for some $r = 1,...,b$ (with the notation for the equivalence classes as above). It suffices to characterize $T_{\theta} N_r^{\pm}$ since $N_r^{\pm}$ descends to a connected component of $\fix(\Phi_\tau)$.  

For the next lemma we identify $S^{2m_r-1}$ with its image under $\iota_r$. As in the flat case, we have a decomposition of $T_\theta TS^{2n-1}$ into a direct sum of horizontal and vertical subspaces: $T_\theta TS^{2n-1} = H(\theta) \oplus V(\theta)$. 
Given $x \in S^{2n-1} \subset \mathbb{R}^{2n}$ we may identify $T_x \mathbb{R}^{2n}$ with $\mathbb{R}^{2n}$. Through this identification we may consider $J_l$ as an endomorphism of  $T_x \mathbb{R}^{2n}$. Given $v \in T_x \mathbb{R}^{2n}$ let $v^T$ denote the tangential projection of $v$ onto $T_x S^{2n-1}$. Then, if we let $J_l^T$ denote the the map $v \mapsto J_l(v)^T$, we may consider $J_l^T$ as an endomorphism of $T_x S^{2n-1}$.   

\begin{lemma}
\label{char}
Let $x \in S^{2m_r-1} \subset S^{2n-1}$. Under the identifications described above:
\begin{equation}
T_\theta \mathcal{J}_l(S^{2m_r-1}) \cong \text{ Graph}\left(\restr{J_l^T}{T_x S^{2m_r-1}} \right)\subset H(\theta) \oplus V(\theta).
\end{equation}
(Recall that $H(\theta)$ and $V(\theta)$ are both isomorphic to $T_xS^{2n-1}$.)
\end{lemma} 
\begin{proof}
Let $\sigma(t) = (\alpha(t), \hspace{2 pt} Z(t))$ be a curve in $TS^{2n-1}$ such that $\sigma(t)$ is contained in $\mathcal{J}_l(S^{2m_r-1})$ and $\sigma(0) = \theta$. By hypothesis $Z(t) = J_l(\alpha(t))$. We compute $\dot{\sigma}(0) \in H(\theta) \oplus V(\theta)$.
The $H(\theta)$-component is given by: 
\begin{eqnarray*}
D_x \pi (\dot \sigma (0)) &=& \restr{\frac{d}{dt}}{t=0} \pi(\sigma(t)) \\
&=& \dot{\alpha}(0).
\end{eqnarray*}  
The $V(\theta)$ component is given by $K_\theta (\dot{\sigma}(0))$. We compute:
\begin{eqnarray*}
K_\theta (\dot{\sigma}(0)) &=& \nabla_{\alpha} Z(0)\\
&=& \left(\overline{\nabla}_{\dot{\alpha}(0)} \tilde{Z}\right)^T,
\end{eqnarray*}
where $\overline{\nabla}$ is the Euclidean connection and $\tilde{Z}$ is an extension of $Z$ to a vector field on $\mathbb{R}^{2n}$. In standard Euclidean coordinates let $\dot{\alpha}(0) = Y^i \frac{\partial}{\partial x^i}$ (with the Einstein summation convention). An extension of $Z$ is given by $\tilde{Z}(x) = J^{ij}_l x_j \frac{\partial}{\partial x_i}$ ($J^{ij}_l$ are the matrix elements of $J_l$ with respect to the standard basis). Now we compute:  
\begin{eqnarray*}
\overline{\nabla}_{\dot{\alpha}(0)} \tilde{Z} &=& Y^k \frac{\partial}{\partial x_k} J^{ij}_l x_j \frac{\partial}{\partial x_i} \\
&=& Y^k J^{ik}_l \frac{\partial}{\partial x_i} \\
&=& J_l(\dot{\alpha}(0))(x). 
\end{eqnarray*}
Therefore $K_\theta (\dot{\sigma}(0)) = J_l^T(\dot{\alpha}(0))$.  
\end{proof}
\begin{proof}[Proof of Proposition \ref{clean}]
Since the geodesic flow on $L(q; p_1,...,p_n)$ is periodic, we need only consider $0< \tau < 2 \pi$ when $q$ is odd, or $0< \tau < \pi$ when $q$ is even. Fix $\tau = \frac{2 \pi}{q} l$ with $l = 1,...,q-1$ in the odd case and $l=1,...,\frac{q}{2}-1$ in the even case. Suppose that $(\ol{x}, \ol{v}) = \overline{\theta} \in \fix (\Phi_{\tau} )$. One inclusion is obvious. We need to prove that $\text{ker} \left[ D_{\ol{\theta}} \Phi_{\tau} - \id \right] \subset  T_{\ol{\theta}} \fix(\Phi_{\tau})$. 

By Lemma \ref{sub_inhom} we may assume that $\overline{\theta} \in D\phi(N_r^+)$ for some $r$ (the case of $N_r^-$ is handled similarly). Using the vertical/horizontal splitting of $T_{\overline{\theta}} TL$, we have \newline $D_{\overline{\theta}} \Phi_\tau (\overline{A}, \overline{B}) = (\overline{Y}(\tau),\nabla \overline{Y}(\tau))$ where $\overline{Y}(t)$ is a Jacobi field along the short geodesic with initial conditions $\overline{Y}(0) = \overline{A} \in T_{\ol{x}}L$ and $\nabla \overline{Y}(0) = \overline{B} \in T_{\ol{x}}L$. We have the following standard observations: $D_{\ol{\theta}} \Phi_{\tau}(\ol{v},0) = (\ol{v},0)$ and $D_{\ol{\theta}} \Phi_{\tau}(0,\ol{v}) = (\tau \ol{v}, \ol{v} )$.  Therefore, we need only consider $(\overline{A},\overline{B}) \in \mye_{\ol{\theta}} \oplus \mye_{\ol{\theta}}$, where 
\begin{equation*}
\mye_{\ol{\theta}} = \{ X \in T_{\ol{x}} L; \hspace{4 pt} \left \langle X, \ol{v} \right \rangle = 0 \}. 
\end{equation*}

Let $\sigma(t)$ be the geodesic on the sphere that satisfies $\sigma(0) = x$ and $\dot{\sigma}(0) = v$. Then, $\sigma$ projects to a geodesic in $L$ based at $\overline{x}$. Since the lift of any Jacobi field in $L$ is a Jacobi field on $S^{2n-1}$, it suffices to show that if $Y(t)$ is a Jacobi field on $S^{2n-1}$ along $\sigma$ satisfying $D\phi \left(\sigma\left( \tau \right), Y\left( \tau \right) \right) = D\phi \left(\sigma\left( 0 \right), Y(0) \right)$ and $D\phi \left(\sigma\left( \tau \right), \nabla Y\left( \tau \right) \right) = D\phi \left(\sigma\left( 0 \right), \nabla Y(0) \right)$, then $(Y(0), \nabla Y(0)) \in T_{\theta} \mathcal{J}_l\left( S^{2m_r-1} \right)$. 

Fix $(A,B) \in \mye_{\theta} \oplus \mye_{\theta} \subset T_{\theta}S(S^{2n-1})$. A Jacobi field with initial conditions \newline $(Y(0), \nabla Y(0) ) = (A, B)$ is of the form: 
\begin{equation*}
Y(t) = \cos(t) A(t) + \sin(t) B(t),
\end{equation*}
where $A(t)$ and $B(t)$ are the parallel translations of $A$ and $B$ along $\sigma(t)$, respectively. Let $P_{x}(t)$ be parallel translation along $\sigma$ up to time $t$. Since $S^{2n-1}$ is a symmetric space, then $P_{x}(t)$ is the derivative of a one parameter subgroup of $\SO (2n)$ whose orbit is the image of the geodesic $\sigma$. Thus, by the linearity of the action, $P_{x}(t) \in \SO (2n)$ for every $t$. 

Now suppose that $Y(t)$ descends to a periodic Jacobi field on $L$ of period $\tau$. This means that $Y(\tau)=  D_{x} T^{k_r} (Y(0) )$ and $\nabla Y(\tau)=  D_{x} T^{k_r} (\nabla Y(0) )$, where $k_r$ realizes the $r$th-equivalence class. This sets up the following equation:
\begin{equation}
\cos(\tau) P_{x}(\tau)A + \sin(\tau) P_{x}(\tau) B = T^{k_r}(A),
\end{equation}
Proceeding as in the proof of Theorem \ref{geos} we write:
\begin{equation}
P_{x}(\tau)B = \frac{1}{\sin(\tau)} \left[T^{k_r} - \cos(\tau)P_{x}(\tau) \right] A.
\end{equation}
We claim that (under the usual identifications) $P_{x}(\tau)B = B$ and $P_{x}(\tau)A = A$. We defer the proof of this fact until after the proof of the proposition. Assuming $P_x(\tau)$ fixes $A$ and $B$, we have:
\begin{equation*}
B = \frac{1}{\sin(\tau)} \left[T^{k_r} - \cos(\tau)\eye \right] A.
\end{equation*}
Similarly, by using the condition on $\nabla Y(t)$ one finds $A$ and $B$ must also satisfy:
\begin{equation*}
A = \frac{1}{\sin(\tau)} \left[T^{k_r} - \cos(\tau)\eye \right] B.
\end{equation*}
Proceeding as in the proof of Theorem \ref{geos} we study the previous two equations for the projections of $A$ and $B$ onto $\myspan_{\mathbb{R}} \left( e_{2i-1}, e_{2i} \right)$. We obtain:
\begin{equation*}
\left[R(\theta_i)^{k_r} -\cos(\tau) \eye \right]^2 B^i = -\sin^2(\tau) B^i,
\end{equation*}
However, as observed earlier, this is true only if $\cos(2 \pi p_i k_r/q ) = \cos( \tau)$. Thus, $Y(t)$ descends to a periodic Jacobi field only if $B = J_l(A)$ with $A \in T_{x}S^{2m_r-1}$. Moreover, since $\langle A, v \rangle = 0$ we see that $\langle J_l(A), x \rangle = \langle A, - J_l(x) \rangle = \langle A, -v \rangle = 0$. Thus $B = J_l(A) = J_l^T(A)$. By Lemma \ref{char} this implies that $(A,B) \in T_{\theta} \mathcal{J}_l(S^{2m_r-1})$.
\end{proof}

\noindent Proposition \ref{clean} is proved, and therefore cleanliness, modulo the following lemma:

\begin{lemma}
\label{little_lemma}
With the notation as above, for every $V \in \mye_{\theta}$ we have $P_{x}(t) V = V$ for every $t$ (under the usual identifications). 
\end{lemma}

\begin{proof}
First call a point $x_0 \in S^{2m_r-1} \subset S^{2n-1}$ a {\it pole} if $x_0 = (0,...,0,1,0,...,0)$. For every $m_r \in \lbrace m_1,...,m_b \rbrace$, $S^{2m_r-1}$ contains at least one pole. Fix a pole $x_0 \in S^{2m_r-1}$. If $P_{x_0}(t)$ denotes parallel translation along the short geodesic in the direction $J_l(x_0)$ then one easily observes that if $X \in T_{x_0} S^{2n-1}$ with $\langle X, J_l(x_0) \rangle = 0$ then $P_{x_0}(t) A = A$.

Given a general $x \in S^{2m_r-1}$, let $g \in \SO (2n)$ with $g(x) =x_0$ and $g(J_l(x)) = J_l(x_0)$. Then $P_{x}(t) = g^{-1} P_{x_0}(t) g$. From this we observe that if $\langle X, J_l(x) \rangle = 0$, then 
\begin{eqnarray*}
P_{x}(t) X &=&  g^{-1} P_{x_0}(t) g X \\
&=& g^{-1} g X \\
&=& X,
\end{eqnarray*}
as required.
\end{proof}

From this we conclude that for homogeneous lens spaces the Poisson relation is an equality.

\begin{proof}[Proof of Corollary \ref{homog_result}]
If $L(q; p_1,...,p_n)$ is homogeneous, then for every $l$ there is only one $\overset{(l)}{\sim}$-class. Thus, the sum in \eqref{zeroth} has two terms, i.e. $\fix(\Phi_{\tau})$ has two connected components coming from the closed geodesics and their time-reversals. Since the two fixed point sets are isometric (via the map that is the antipodal map in each fiber of $SL$), and the Morse index and Duistermaat-Guillemin density is invariant under time reversal (see \cite[Appendix]{uribe}), the two terms are the same. Thus the sum in \eqref{zeroth} for a homogenous lens space is non-zero.     
\end{proof}

For lens spaces that are not homogeneous there can be more terms in \eqref{zeroth}. Therefore, a priori, cancellation is possible. A first step in showing that cancellations do not occur is computing the Morse index of a closed geodesic. 

\section{Computation of the Morse Index}
\label{morse_index_comp}

We briefly recall the formula for the Morse index of a closed geodesic (following \cite{ballmann_closed_geos}).  Fix $\sigma\colon \left[0, \tau \right] \to M$ a closed geodesic in a closed Riemannian manifold $(M,g)$ of length $\tau$ parametrized by arclength. Let $V_{\Lambda} (\sigma )$ be the space of piecewise $C^{\infty}$-vector fields $X$ along $\sigma$ satisfying $\langle X(t), \dot{\sigma}(t) \rangle = 0$. On this vector space we have the usual index form for the energy functional $H\colon V_{\Lambda} (\sigma ) \times V_{\Lambda} (\sigma ) \to \mathbb{R}$ given by:
\begin{eqnarray*}
H(X,Y) = \int_0^{\tau} \langle \nabla X, \nabla Y \rangle - \langle R(X, \dot{\sigma}) \dot{\sigma}, Y \rangle dt.
\end{eqnarray*}
The Morse index of $\sigma$ will be denoted by $\ind(\sigma)$ and is the index of $H$. The nullity of $\sigma$, denoted by $\mynull(\sigma)$, is $\dim \ker \left(\restr{H}{V_{\Lambda} (\sigma ) } \right)$. We also define $\ind_{\Omega}(\sigma)$ to be the index of $H$ on the space of fixed endpoint variations: $V_{\Omega}(\sigma)$, which is the collection of piecewise smooth vector fields on $\sigma$ that are orthogonal to $\dot{\sigma}$ and satisfy $X(0) = X(\tau) = 0$. We define $\mynull_{\Omega}(\sigma)$ similarly. 

Let $\mu(t)$ be the number of linearly independent Jacobi fields $Y(t)$ along $\sigma$ such that $Y(0) = Y(\tau) = 0$. Then the {\it Morse Index Theorem} states (see \cite{milnor_morse_theory}) that:
\begin{equation*}
\mynull_{\Omega}(\sigma) = \mu(\tau),
\end{equation*}
and
\begin{equation*}
\ind_{\Omega}(\sigma) = \sum_{0<t<\tau} \mu(t).
\end{equation*}
To compute the index for a variation that does not have fixed endpoints, we recall the following linear algebra fact: Let $H$ be a symmetric bilinear form on a finite dimensional real vector space $V$. For any subspace $W \subset V$ we have:
\begin{equation}
\label{indexeq}
\ind H = \ind \restr{H}{W} + \ind \restr{H}{W^{\perp}} + \dim(W \cap W^{\perp}) - \dim(W \cap \ker H),
\end{equation}
where $W^{\perp} = \lbrace X \in V; \hspace{4pt} H(X,Y) = 0 \text{ for all } Y \in W \rbrace$. 
To use (\ref{indexeq}), we fix $0 = t_0<t_1<...<t_k = \tau$ such that $\restr{\sigma}{[t_i,t_{i+1}]}$ has no conjugate points. Then $V_{\Lambda}(\sigma)$ is a direct sum of $V_{\Lambda}^1(\sigma)$ and $V_{\Lambda}^2(\sigma)$ where:
\begin{equation*}
V^1_{\Lambda}(\sigma) = \lbrace X \in V_{\Lambda}(\sigma); \hspace{4 pt} \restr{X}{[t_i, t_{i+1}]} \text{ is a Jacobi field } \rbrace,
\end{equation*}
and
\begin{equation*}
V^2_{\Lambda}(\sigma) = \lbrace X \in  V_{\Lambda}(\sigma); \hspace{4 pt} X(t_i) = 0 \text{ for every } i \rbrace. 
\end{equation*}
Observe that $V^1_{\Lambda}( \sigma)$ and $V^2_{\Lambda}(\sigma)$ are orthogonal with respect to $H$, and $H$ is positive definite on $V^2_{\Lambda}(\sigma)$. Therefore, $\ind(\sigma) = \ind \restr{H}{V^1_{\Lambda}(\sigma)}$ (hence is finite). Now set $W = V^1_{\Lambda}(\sigma) \cap V_{\Omega}(\sigma)$. Then $\ind \restr{H}{W} = \ind_{\Omega}(\sigma)$. One observes that $W^{\perp}$ is the set of smooth Jacobi fields $Y(t)$ along $\sigma$ with $Y(0) = Y(\tau)$; however, it is not necessary that $\nabla Y(0) = \nabla Y(\tau)$. Thus, $\dim(W \cap W^{\perp}) = \mu(\tau)$. One also observes that $W \cap \text{ker}(H)$ is the collection of periodic Jacobi fields vanishing at 0. Moreover, from the formula for the index form above, it is clear that if $X$, $Y \in W^{\perp}$:
\begin{equation*}
H(X,Y) = \langle \nabla X(\tau^-) - \nabla X(0^+), Y(0) \rangle.
\end{equation*}
Suppose that $\sigma(0) = x$ and $\dot{\sigma}(0) = v$. Then, if $\theta = (x, v) \in TM$, using the geodesic flow and the standard symplectic form $\Omega$ on $TM$, we write:
\begin{equation*}
H(X,Y) = - \Omega \left( (D_{\theta} \Phi_{\tau} - \id)(X(0), \nabla X(0)), (Y(0),\nabla(Y(0)) \right),
\end{equation*}
for $X,Y \in W^{\perp}$.
Define the {\it concavity form} on $\left(D_{\theta} \Phi_{\tau} - \id \right)^{-1}(0, \mye_{\theta})$ by:\footnote{Since $X(t)$ and $Y(t)$ must be orthogonal to $\dot{\sigma}(t)$, then the initial conditions for $X$ and $Y$ must be in $\mye_{\theta} \oplus \mye_{\theta}$.} 
\begin{equation*}
\tilde{H}(X,Y) = -\Omega ((D_{\theta} \Phi_{\tau} - \id)X,Y).
\end{equation*}
Thus, $\ind \restr{H}{W^{\perp}} = \ind \tilde{H}$. Using Equation (\ref{indexeq}) we have the following theorem:

\begin{theorem}[\cite{ballmann_closed_geos} Section 1]
\label{morse_generalization}
With the notation as above, if $\sigma$ is a smoothly closed geodesic then:
\begin{equation}
\label{indexeq2}
\ind(\sigma) = \ind_{\Omega}(\sigma) + \mu(\tau) - \dim \left[ (0,\mye_{\theta}) \cap \ker \left( D_{\theta} \Phi_{\tau} - \id \right) \right] + \ind \tilde{H}.
\end{equation}
\end{theorem}
\subsection{The Morse Index of Closed Geodesics on Lens Spaces}

First we compute the Morse index for a closed geodesic of length $2 \pi k$ (if $q$ is odd) or $k \pi$ (if $q$ is even) with $k = 1,2,...$. Let $\sigma$ denote the geodesic. Then in the odd (respectively even) case $\sigma$ lifts to a closed geodesic on $S^{2n-1}$ (respectively $\mathbb{R}\myP^{2n-1}$). Moreover, in both cases the geodesic flow fixes all of the unit tangent bundle. From this we conclude that the index of the concavity form is identically 0, and the third term in \eqref{indexeq2} is equal to $2n-2$. Thus:
\begin{equation*}
\ind(\sigma) = \begin{cases} (2n - 2)(2k - 1) &\mbox{if $q$ is odd} \\ 
(2n - 2)(k-1) & \mbox{if $q$ is even} \end{cases}.
\end{equation*}
In this case, it is clear that the Morse index only depends on the length of the closed geodesic.

Now, we consider lengths coming from short geodesics and their iterates: Let $ \tau$ be the length of a closed geodesic on a lens space $L = L(q; p_1,..., p_n)$ that cannot be lifted to a closed geodesic on the Riemannian cover. Then $\tau = 2\pi k + \frac{2 \pi}{q}l$ for some $k \in \mathbb{N}$ and $l = 1,...,q-1$ in the odd case, and $\tau = \pi k + \frac{2 \pi}{q}l$ for some $k \in \mathbb{N}$ and $l=1,...,\frac{q}{2}-1$ in the even case. Let $\lbrace p_1 = p_{j_1},..., p_{j_{m_1}}; ...; p_{j_{m_{b-1}+1}},...,p_{j_{m_b}} \rbrace$ be a partition of $\lbrace p_1,...,p_n \rbrace$ into equivalence classes with respect to $\overset{(l)}{\sim}$. Moreover let $\overline{\theta} \in \fix (\Phi_{\tau})$ be contained in a connected component which corresponds to an $\overset{(l)}{\sim}$-class that is realized by $k$. 

We compute each term on the right hand side of \eqref{indexeq2}. One readily sees that $\ind_{\Omega}(\sigma)$ depends only on $\tau$ (using the Morse index theorem):
\begin{equation}
\ind_{\Omega}(\sigma) = (2n-2) \left( \floor{\frac{\tau}{\pi}}-1 \right).
\end{equation}
Similarly, $\mu(\tau) = 0$. By Lemma \ref{char} and Proposition \ref{clean} we see that $\ker \left[ D_{\overline{\theta}} \Phi_{\tau} - \id \right]$ is the graph of a linear function, therefore $\dim \left[ (0, \mye_{\overline{\theta}}) \cap \ker \left[ D_{\overline{\theta}} \Phi_{\tau} - \id \right] \right] = 0$. This discussion shows that the first three terms in the right hand side of \eqref{indexeq2} only depend on $\tau$.

Now we compute the index of the concavity form. As above, let $\theta = (x,v) \in S(S^{2n-1})$ be a lift of $\overline{\theta}$ and let $\sigma$ be the lift onto $S^{2n-1}$ of the short geodesic in $L$. We want to find conditions on $X$, a Jacobi field along $\sigma$ such that $\tilde{H}(\overline{X},\overline{X})<0$. Thus we want to find when $\langle \nabla \overline{X}(0),\overline{X}(0) \rangle > \langle \nabla \overline{X}(\tau), \overline{X}(0) \rangle$. Let $(A,B)$ be initial conditions for $X$. Since $X$ is a Jacobi field on the sphere we have:
\begin{equation}
\label{jacobifield1}
X(\tau) = \cos(\tau) P_{x}(\tau) A + \sin(\tau) P_{x}(\tau) B.
\end{equation}
By hypothesis, $X(0)$ and $X(\tau)$ are identified in $L$. Therefore:
\begin{equation}
\label{jacobifield2}
B = \frac{1}{\sin(\tau)} \left[ T^{k} - \cos(\tau) \id \right] A.
\end{equation}
We also have:
\begin{equation}
\label{jacobifield3}
\nabla X (\tau) = -\sin(\tau) P_{x}(\tau) A + \cos(\tau) P_{x}(\tau) B. 
\end{equation}
Using (\ref{jacobifield1})-(\ref{jacobifield3}) we find that for $X$ to satisfy $\tilde{H}(\overline{X},\overline{X}) <0$ we must have:
\begin{eqnarray*}
\frac{1}{\sin(\tau)}\langle T^{k} A, A \rangle - \frac{\cos(\tau)}{\sin(\tau)} |A|^2 >&-\sin(\tau)\langle T^{-k} A, A \rangle + \cos(\tau) \langle T^{-k} B,A \rangle \\
=&-\sin(\tau)\langle A, T^{k} A \rangle +  \cos(\tau) \langle B, T^{k} A \rangle \\
=&-\sin(\tau)\langle A, T^{k} A \rangle + \frac{\cos(\tau)}{\sin(\tau)} \langle T^{k} A, T^{k}A \rangle \\ 
&-\frac{\cos^2(\tau)}{\sin(\tau)}\langle A, T^{k}A \rangle \\ 
\end{eqnarray*} 
Which is equivalent to:
\begin{equation}
\label{condition}
\frac{1}{\sin(\tau)} \langle T^{k}A, A \rangle >  \frac{\cos(\tau)}{\sin(\tau)} |A|^2.
\end{equation}

For ease of computation, let us assume that $x$ is a vector of zeros with 1 in the $j$-th slot. Then $\mye_{\theta} \subset T_x S^{2n-1}$ has the following orthonormal basis:  
\begin{equation*}
\lbrace e_1, J(e_1),..., \hat{e}_j,\widehat{J(e_j)},..., e_{2n}, J(e_{2n}) \rbrace, 
\end{equation*}
where $e_i$ is the standard basis vector in $\mathbb{R}^{2n}$ and $\hat{v}$ denotes the fact that $v$ is missing from the collection. Computing with respect to this basis, the following proposition is clear: 

\begin{proposition}
\label{index}
Let $\tau$ be the length of a short geodesic and $l$ as above. Fix a partition of $\lbrace p_1,..., p_n \rbrace$ into $\overset{(l)}{\sim}$-equivance classes. Let $(\overline{x}, \overline{v}) = \overline{\theta} \in \fix (\Phi_{\tau})$ with $\overline{x}$ $\overset{(l)}{\sim}$-adapted corresponding to the $r$th-equivalence class. Moreover, suppose that $k_r$ realizes this equivalence class. Then the index of the concavity form for a geodesic based at $\overline{x}$ in the direction $\overline{v}$ is given by:
\begin{equation}
\label{concavity}
2  \# \left \lbrace p_i ; \hspace{4 pt} \frac{\cos\left( 2 \pi  k_r p_i/q \right)}{\sin(\tau)} > \frac{\cos(\tau)}{\sin(\tau)} \right \rbrace.
\end{equation}
\end{proposition}

\begin{remark}
Since the only other choice of $k_r$ that realizes a fixed equivalence class is the negation (modulo $q$) of the original, it is clear that the index of the concavity form does not depend on the choice of $k_r$ (and therefore does not depend on the orientation of the geodesic). 
\end{remark}

From Proposition \ref{index} it is clear that the Morse index depends not only on the length of the closed geodesic, but on the component of $\fix(\Phi_{\tau})$ containing $\overline{\theta}$. Thus, the terms in \eqref{zeroth} can in general have opposite signs, which means cancellation could still occur. This is not the case for the systole of $L(q; p_1,...,p_n)$ (which in this case corresponds to the length of the shortest closed geodesic).

\begin{proof}[Proof of Theorem \ref{systole_result}]
Set $\tau = \frac{2 \pi}{q}$.   
Let $\lbrace \Theta_j^{\pm} \rbrace_{j=1}^{r} \subset \fix (\Phi_{\tau})$ be the fixed point sets of largest dimension, with $\pm$ denoting the orientation of the geodesics contained in each set. We need to show that the following sum is non-zero:
\begin{equation*}
\sum_{j=1}^r \frac{i^{- \sigma_j}}{2 \pi} \left(\int_{\Theta_j^+} \text{ d}\mu_j+ \int_{\Theta_j^-} \text{ d}\mu_j \right),
\end{equation*}
where $\sigma_j$ denotes the Morse index corresponding to a closed geodesic loop with initial conditions given by $\overline{\theta}_j \in \Theta_j^{\pm}$. It suffices to show that the Morse indices are equal. By \eqref{indexeq2} and the above discussion, the only contributing term towards the Morse index that could differ depending on $\overline{\theta}_j$ is the index of the concavity form. 

For fixed $j$ we know that if $(\overline{x}_j, \overline{v}_j) = \overline{\theta}_j$, then $\overline{x}_j$ is $\overset{(1)}{\sim}$-adapted corresponding to some equivalence class. If $k$ realizes this equivalence class, then for any $p_i$ in the equivalence class we have: $\cos(2 \pi k p_i/q) = \cos(2 \pi /q)$. Notice that since $1$ and $p_i$ are both units in the ring $\mathbb{Z}_q$, $k$ must be a unit also. This means that for any other $p_l \overset{(1)}{\nsim} p_i$ we know that $k p_l \not\equiv 0$ (mod $q$). By \eqref{concavity}, this implies that the index of the concavity form for $\overline{\theta}_j$ is zero. Thus, the Morse index does not depend on which component $\overline{\theta}_j$ is from. Indeed, $\sigma_j =0$, independent of $j$. 
\end{proof}

So far, we have computed the singular support of the wave trace for homogeneous lens spaces. The theorem above demonstrates that the systole is always in the singular support of the wave trace on an arbitrary lens space. However, since the Morse index in general depends on more than just the length of the corresponding geodesic, to rule out cancellation in \eqref{zeroth} we must compute the Duistermaat-Guillemin density. 

\section{Computation of the Duistermaat-Guillemin density}
\label{DG_density_section}

In this section we construct the Duistermaat-Guillemin density (DG-density), first constructed in \cite[Section 4]{duistermaat_spec}. Let $(M,g)$ be a compact Riemannian manifold and assume that $M$ is clean. Let $\tau \in \specl(M,g)$ and let $\Theta \subset \fix(\Phi_{\tau})$ be connected. Let $\tilde{\Theta} = \{ t X_x ; \hspace{4 pt} X_x \in \Theta \text{ and } t>0 \}$. Using the symplectic form on $TM$ we construct a canonical density $\tilde{\mu}^{\tau}$ on $\tilde{\Theta}$ and obtain a canonical density on $\Theta$ by dividing by $\text{dvol}$ in the transverse direction. 

First, notice that $\tilde{\Theta}$ is a clean fixed point set of $\Phi_{\tau}$ (in $TM$). Let $\theta \in \tilde{\Theta}$ and $P_\theta = \eye - D_{\theta} \Phi_{\tau} \colon V \to V$ where $V = T_{\theta}TM$. Following the construction of the DG-density given in the appendix of \cite{uribe}, we form:
\begin{itemize}
\item $E = \{e_1,...,e_k \}$, a basis for $W = T_{\theta} \tilde{\Theta}$;
\item Let $W^{\Omega}$ be the symplectic complement of $W$ in $V$;
\item Let $F = \{f_1,...,f_k\}$ be a basis for the complement of $W^{\Omega}$ satisfying
\begin{equation*}
\Omega(e_i,f_j) = \delta_{ij};
\end{equation*}
\item Let $V = \{v_1,...,v_{2n-k}\}$ be a basis for the complement of $W$ in $V$.
\end{itemize}
The following lemma gives a formula for the DG density with respect to this basis.
\begin{lemma}[\cite{uribe} Section A.1]
\label{uribe_dg}
With the notation as above, $\text{ker}(P) = W$ and the image of $P$ is $W^{\Omega}$. So $P(V) \oplus F$ is a basis for $V$. If $\nu$ is a $\frac{1}{2}$-density on $T_{\theta} TM$, then the DG-density $\tilde{\mu}^{\tau}$ on $\tilde{\Theta}$ is given by:
\begin{equation}
\label{uribe_eq}
\tilde{\mu}^{\tau}(e) = \frac{\nu(v \wedge e)}{\nu(Pv \wedge f)},
\end{equation}
where $\displaystyle e = \bigwedge_{i} e_i$. The expression is independent of $f,v,$ and $\nu$. Moreover, $\tilde{\mu}^{\tau} = \tilde{\mu}^{-\tau}$. 
\end{lemma} 

Throughout this section fix a length $\tau = \frac{2 \pi}{q}l$ of a closed geodesic on the lens space $L(q; p_1,...,p_n)$ (since the geodesic flow is periodic, it suffices to consider $0< \tau < 2 \pi$ [$\pi$ in case $q$ is even]). Let $\lbrace p_1 = p_{j_1},..., p_{j_{m_1}}; ...; p_{j_{m_{b-1}+1}},...,p_{j_{m_b}} \rbrace$ be a partition of $\lbrace p_1,...,p_n \rbrace$ into equivalence classes with respect to $\overset{(l)}{\sim}$. Let $\Theta \subset \fix(\Phi_{\tau})$ be a connected submanifold. Let $\overline{\theta} = (\ol{x}, \ol{v}) \in \Theta$. By Corollary \ref{gen_length} we know that $\ol{x}$ is $\overset{(l)}{\sim}$-adapted corresponding to some equivalence class. Let $p_*$ be a representative for this equivalence class. Moreover, suppose that the equivalence class corresponding to $\ol{x}$ has $m$ elements and is realized by $k$. With this notation, we have:

\begin{theorem}
\label{DG_density}
The Duistermaat-Guillemin density on $\Theta$ is given by:
\begin{equation}
\mu = \frac{1}{\sqrt{\tau}}2^{-(m-1)/4} |\sin(\tau)|^{-(m-1)/2} \left(\prod_{p_i \overset{(l)}{\nsim} p_*} 2|\cos(2 \pi k p_i/q) - \cos(\tau) | \right)^{-1} |\dvol|,
\end{equation}
where $|\dvol|$ is the Riemannian density (induced from the Sasaki metric) on $\Theta$. 
\end{theorem}

\begin{proof}
It suffices to do the calculation on the sphere covering $L$. Let $\theta = (x, v) \in TS^{2n-1}$ be a lift of $\ol{\theta}$. Since $\ol{x}$ is $\overset{(l)}{\sim}$-adapted corresponding to some equivalence class, $x$ is in the image of $\iota \colon S^{2m-1} \to S^{2n-1}$. Assume without loss of generality that the image of $\iota$ is in the first $2m$ coordinates of $\mathbb{R}^{2n}$. We identify $S^{2m-1}$ with its image under $\iota$. 

We construct a local section of the restriction of the orthonormal frame bundle over $S^{2n-1}$ to $S^{2m-1}$. For simplicity of notation, we suppress the dependence on the base point. Let  $O = \left[ b_1,..., b_{2m-1} \right]$ be an orthonormal frame around $x$ such that $b_1 = J_{l}(x)$ (here we think of each $b_i$ as an element of $\mathbb{R}^{2n}$). Let $W = \left[ v_{2m+1},...,v_{2n} \right]$ be such that $O \oplus W$ is a full orthonormal frame for $S^{2n-1}$ (locally defined in an $S^{2m-1}$-neighborhood of $x$). By our assumption on $\iota$, we may take $v_i$ to be the $i$-th standard basis vector in $\mathbb{R}^{2n}$ (independent of the base point in $S^{2m-1}$). 

We use this orthonormal frame to construct a local orthonormal frame around $\theta$ in $TS^{2n-1}$, with respect to the Sasaki metric. Using the vertical/horizontal splitting, we define the following orthonormal frames around $\theta$: 
\begin{align}
E &= \left[ (b_1, 0),  \frac{1}{\sqrt{2}}(b_2,J_{l}(b_2)),...,\frac{1}{\sqrt{2}}(b_{2m-1},J_l(b_{2m-1})) \right]\\
B_1 &= \left[(v_{2m+1},0),...,(v_{2n},0), (0, v_{2m+1}),...,(0,v_{2n}) \right]\\
B_2 &= \left[(0,b_1),  \frac{1}{\sqrt{2}}(J_{l}(b_2),b_2),...,\frac{1}{\sqrt{2}}(J_l(b_{2m-1}),b_{2m-1}) \right]. 
\end{align}
By Lemma \ref{char}, $E$ spans $T_\xi \mathcal{J}_l(S^{2m-1}) \subset T_{\xi}TS^{2n-1}$ for $\xi$ in a $\mathcal{J}_l(S^{2m-1})$-neighborhood of $\theta$ (here we are assuming that geodesics in $\Theta$ are oriented positively). Also $E \oplus B_1 \oplus B_2$ spans $T_\xi TS^{2n-1}$.
Define the following frame that is complementary to the symplectic orthogonal of $E$:
\begin{equation}
F = \left[ (0,b_1), \frac{1}{\sqrt{2}}(-J_{l}(b_2),b_2),...,\frac{1}{\sqrt{2}}(-J_l(b_{2m-1}),b_{2m-1}) \right].
\end{equation}
Simple computations verify that $F$ and $E$ satisfy: $\Omega(e_i,f_j) = \delta_{ij}$. We set $V = B_1 \oplus B_2$. Then $V$ is clearly a basis for the complement of $E$.   

Let $\nu = |\dvol|^{\frac{1}{2}}$ be the $\frac{1}{2}$-density induced from the Riemannian density on $TS^{2n-1}$ (with the Sasaki metric). A simple computation verifies that $E \oplus V$ is an orthonormal basis of $T_{\xi}TS^{2n-1}$.

We want to compute the expression in \eqref{uribe_eq} using $E$,$V$, and $F$ above. Clearly the numerator is 1. Thus it remains to compute the denominator of \eqref{uribe_eq}. Once again, using the fact that $TS^{2n-1} \overset{D\phi}{\to} TL$ is a Riemannian cover it suffices to compute the denominator of \eqref{uribe_eq} using the lift of the Poincare map to $TS^{2n-1}$. Given $(y, w) = \xi \in TS^{2n-1}$, define $DT \colon TS^{2n-1} \to TS^{2n-1}$ to be the induced action of $T \in \SO(2n)$ on $TS^{2n-1}$ given by $DT(y,w) = (Ty, Tw)$. Then using the linearity of the action (and the fact that $T$ acts via isometries), it is clear that $D_{\xi}DT(A,B) = (TA,TB)$ for $(A,B) \in T_{\xi}TS^{2n-1}$ (using the vertical/horizontal splitting). Then we need to understand the image of $V$ under $\hat{P} = \eye - D_{\theta} (DT^{-k} \circ \hat{\Phi}_{\tau})$, where $\hat{\Phi}_{t}$ is the time-$t$ map of the geodesic flow on $T(S^{2n-1})$. Observe that $\hat{P}$ preserves the 4-dimensional subspace spanned by $\lbrace (v_{2i-1},0),(v_{2i},0),(0,v_{2i-1}),(0,v_{2i}) \rbrace$, $i=1,...,n$.  Thus it suffices to study $\hat{P} = \eye - D_{\theta} (DT^{-k} \circ \hat{\Phi}_{\tau})$ on the invariant subspace spanned by $\lbrace (v_{2m+2i-1},0), (v_{2m+2i},0), (0,v_{2m+2i-1}), (0,v_{2m+2i})\rbrace$, $i=1,..., n-m$. 

It will often be easier from a computational standpoint to instead evaluate \newline $D_{\theta} DT^k(\eye - D_{\theta} (DT^{-k} \circ \hat{\Phi}_{\tau})) = D_{\theta}DT^k - D_{\theta} \hat{\Phi}_{\tau} =: \tilde{P}$, which still preserves the 4-dimensional subspaces above. We have $\hat{P} = D_{DT^k\theta}DT^{-k} \circ \tilde{P}$. It is easily seen that, restricted to this invariant subspace and with respect to this basis, $D_{\theta}DT^k - D_{\theta} \hat{\Phi}_{\tau}$ has the form:
\begin{equation}
\label{V_mat}
\newcommand*{\tempa}{\multicolumn{1}{|c}{-\sin(\tau) \eye}}
\newcommand*{\tempb}{\multicolumn{1}{|c}{\sin(\tau) \eye}}
\newcommand*{\tempc}{\multicolumn{1}{|c}{R(\alpha_i)^k - \cos(\tau) \eye}}
 \left(D_{\theta}DT^k - D_{\theta} \hat{\Phi}_{\tau}  \right)(v) = \left[
\begin{array}{cc}
R(\alpha_i)^k - \cos(\tau)\eye &\tempa \\ \hline
\sin(\tau) \eye &  \tempc
\end{array}\right]v,
\end{equation}
where $\alpha_i = \frac{2 \pi}{q}p_i$. A computation shows that the determinant of the matrix in \eqref{V_mat} is $4 (\cos(2 \pi k p_i/q) - \cos(\tau))^2$. Therefore the determinant of $\hat{P}$
restricted to one of these 4-dimensional invariant subspaces is $4 (\cos(2 \pi k p_i/q) - \cos(\tau))^2$.

Now we consider $\tilde{P}$ acting on $B_2$. First notice that $\hat{P}(0,b_1) = (\tau b_1,0)$. Moreover, by construction, the horizontal component of $\tilde{P}(\frac{1}{\sqrt{2}} (J_l(b_i),b_i))$ is given by (applying Lemma \ref{little_lemma}):
\begin{align*}
\frac{1}{\sqrt{2}} \left[T^k J_l(b_i) - \cos(\tau)J_l(b_i) - \sin(\tau) b_i \right] &= \frac{1}{\sqrt{2}}(T^k - \cos(\tau)) J_l(b_i) - \frac{\sin(\tau)}{\sqrt{2}} b_i \\
&= \frac{\sin(\tau)}{\sqrt{2}} J_l^2 b_i - \frac{\sin(\tau)}{\sqrt{2}} b_i\\
&= \frac{-2}{\sqrt{2}} \sin(\tau) b_i.
\end{align*}
Similarly, one can show that the vertical component is given by:
\begin{equation*}
\frac{2}{\sqrt{2}} \sin(\tau) J_l(b_i).
\end{equation*}
Thus, 
\begin{equation}
\hat{P} \left(\frac{1}{\sqrt{2}} (J_l(b_i),b_i) \right) = \frac{2}{\sqrt{2}} \sin(\tau) (-T^{-k}b_i, T^{-k}J_l(b_i) ).
\end{equation}

We now compute the denominator of \eqref{uribe_eq}. We have: 
\begin{align}
\nu(\hat{P}v \wedge f) &= \nu\left( \hat{P}b^1 \wedge \tau (b_1,0) \wedge 2 \sin(\tau) \tilde{b^2} \wedge f \right), 
\end{align}
where $b^1$ is the wedge product of elements in $B_1$, and $\tilde{b^2}$ is the wedge product of elements in:
\begin{equation*}
\tilde{B_2} = \left[  \frac{1}{\sqrt{2}}(-b_2,J_{l}(b_2)),...,\frac{1}{\sqrt{2}}(-b_{2m-1},J_l(b_{2m-1})) \right].
\end{equation*}
Since $\nu(b^1 \wedge (b_1,0) \wedge \tilde{b^2} \wedge f) = 1$, we have:
\begin{align}
\nu(\hat{P}v \wedge f)  &= \sqrt{\tau} \left(\prod_{p_i \overset{(l)}{\nsim} p_*} 2|\cos(2 \pi k p_i/q) - \cos(\tau) | \right) \left( \prod_{p_j \overset{(l)}{\sim} p_*} 2^{1/4} |\sin(\tau)|^{1/2}\right)\\
&= \sqrt{\tau} 2^{(m-1)/4} |\sin(\tau)|^{(m-1)/2} \left(\prod_{p_i \overset{(l)}{\nsim} p_*} 2|\cos(2 \pi k p_i/q) - \cos(\tau) | \right),
\end{align}
which completes the proof.
\end{proof}

\section{Three-dimensional lens spaces}
\label{3-dcase}

In this section we study the special case when the lens space is of the form $L(q; p_1,p_2)$ with $q$ an odd prime, i.e. the lens space is 3-dimensional and has fundamental group of odd prime order. The dimensional restriction provides the simplest case in which cancellation can occur in \eqref{zeroth}. First we prove a number-theoretic lemma. 
\begin{lemma}
\label{number_theory}
Let $q$ be an odd prime with $p \in \lbrace 1,..., q-1 \rbrace$. Fix $l \in \lbrace 1,2,...,q-1 \rbrace$ and let $k \in \lbrace 1,2,...,q-1 \rbrace$ be chosen so that $\cos( 2 \pi k p/q) = \cos(2 \pi l/q)$. Then
\begin{equation}
\label{relation}
\cos(2 \pi p l/q) - \cos(2 \pi p k/q ) - \cos(2 \pi l/q) + \cos(2 \pi k/q) = 0,
\end{equation}
if and only if $p \equiv \pm 1$ (mod $q$). 
\end{lemma}

\begin{proof} The equation clearly holds if $p \equiv \pm 1$ (mod $q$). 

For the other direction, 
let $\alpha \in \mathbb{Z}[\zeta_q]$ be an algebraic integer in $\mathbb{Q}(\zeta_q)$ where $\zeta_q = \exp(2 \pi i /q)$. Since $q$ is an odd prime we may write $\alpha$ uniquely as:
\begin{equation}
\alpha = \sum_{i=1}^{q-1} \epsilon_i \zeta_q^i, 
\end{equation}
for $\epsilon_i \in \mathbb{Z}$. Suppose that $\re (\alpha) = 0$. Then we see that $\alpha$ must be of the form:
\begin{equation}
\alpha = \sum_{i=1}^{\frac{q-1}{2}} \epsilon_i \left( \zeta_q^i - \zeta_q^{-i} \right),
\end{equation}
that is, $\alpha$ is an integral sum of sines.

If \eqref{relation} holds, then we would have the following algebraic integer: $\alpha = \zeta_q^{pl}- \zeta_q^{pk} + \zeta_q^k - \zeta_q^l$, with $\re (\alpha) = 0$. Since $\cos(2 \pi k p /q) = \cos(2 \pi l /q)$, we must have $p k \equiv \pm l$ (mod $q$). If $p k \equiv -l$ (mod $q$), then we have: $\alpha = \zeta_q^{pl} - \zeta_q^{-l}+\zeta_q^k - \zeta_q^l$. However, $\alpha$ must be of the form given in \eqref{relation} so $p \equiv 1$, and $l\equiv \pm k$ (mod $q$) with $\alpha = 0$. 

If $pk \equiv l$ (mod $q$), then we have: $\alpha = \zeta_q^{pl} - \zeta_q^{l}+\zeta_q^k - \zeta_q^l$. So, consulting \eqref{relation}, $pl \equiv k$ (mod $q$) and $k \equiv -l$ (mod $q$). Thus $p \equiv -1$ (mod $q$). 
\end{proof}

\begin{remark}
The author is grateful to Tom Shemanske for supplying the proof Lemma \ref{number_theory}.
\end{remark}

We leverage Lemma \ref{number_theory} to prove Theorem \ref{sporadic_result}:

\begin{proof}[Proof of Theorem \ref{sporadic_result}]
We need to show that the for every $\tau \in \specl \left(L(q;p_1,p_2) \right)$ the sum of the zeroth-order wave invariants corresponding to $\tau$ is nonzero. Fix $\tau = \frac{2 \pi l}{q}$, the length of a closed geodesic. It suffices to assume that the geodesic does not lift to a closed geodesic on $S^3$, since that case was handled in Section \ref{morse_index_comp}. There are two cases: either $p_1 \overset{(l)}{\sim} p_2$ or $p_{1} \overset{(l)}{\nsim} p_2$. 
\newline
\noindent {\it Case 1: $p_1 \overset{(l)}{\sim} p_2$}.
\newline In this case $\fix(\Phi_{\tau})$ has two isometric connected components corresponding to the closed geodesics and their time-reversals. Therefore, the Morse index and DG-density are identical in this case, and the two terms add constructively.
\newline
\noindent {\it Case 2: $p_1 \overset{(l)}{\nsim} p_2$}.
\newline
In this case let $S^1_i$ be the image of the isometric embedding of $S^1$ into $S^3$ which corresponds to the lift of $\overset{(l)}{\sim}$-adapted points for $p_i$. Set $N_i^{\pm} = \mathcal{J}_l^{\pm}(S^1_i)$. Then $\fix(\Phi_{\tau}) = D\phi(N_1^+) \cup D\phi(N_2^+) \cup D\phi(N_1^-) \cup D\phi(N_2^-)$ (the union is disjoint). As before, since $N^+_i$ is isometric to $N^-_i$ and the DG-density is invariant under time reversal, there are two terms in the sum in \eqref{zeroth}:
\begin{align}
\label{2dcase}
 \frac{i^{- \sigma_1}}{\pi} \int_{D\phi (N_1^+)} \mu_1 + \frac{i^{- \sigma_2}}{\pi} \int_{D\phi (N_2^+)} \mu_2. 
\end{align}
Lastly, assume that $k \in \lbrace 1,2,...,q-1 \rbrace$ realizes $l$ for $p_2$. By Proposition \ref{index}, the only way for the sum in \eqref{zeroth} to possibly be zero is if $\cos \left(\frac{2 \pi}{q} p_2 l \right) > \cos \left(\tau \right)$, and $\cos(\tau) > \cos \left( \frac{2 \pi}{q} k \right)$, or vice versa. 

Assume without loss of generality that $\cos \left(\frac{2 \pi}{q} p_2 l \right) > \cos \left(\tau \right)$ and \newline $\cos(\tau) > \cos \left( \frac{2 \pi}{q} k \right)$. Then the two terms in \eqref{2dcase} have opposite signs. By Theorem \ref{DG_density}, they cancel if and only if:
\begin{equation}
\label{not_possible}
\cos(2 \pi l p_2/q) - \cos(2 \pi k p_2/q) = \cos(2 \pi l/q) - \cos(2 \pi k/q).
\end{equation}
By Lemma \ref{number_theory}, \eqref{not_possible} holds only if $p_2 \equiv \pm 1$ (mod $q$), which means $p_2 \overset{(l)}{\sim} 1$. Thus the sum in \eqref{zeroth} is non-zero.    
\end{proof}

\bibliography{mybib}{}
\bibliographystyle{alpha}

\end{document}